\documentclass[10pt]{article}

\usepackage{bigints}
\pdfoutput=1

\usepackage{amssymb}
\usepackage{mathrsfs}
\usepackage{amsmath}
\usepackage{amsfonts}

\usepackage{amsthm}
\newtheorem{theo}{Theorem}[section]

\newtheorem{coro}{Corollary}[section]
\newtheorem{lemm}{Lemma}[section]
\newtheorem{fact}{Fact}[section]
\newtheorem*{main}{Main Theorem}

\theoremstyle{definition}
\newtheorem{defi}{Definition}[section]

\theoremstyle{remark}

\newtheorem{rema}{Remark}[section]

\usepackage{hyperref}

\begin{document}

\title{\large{{\bf Decompositions of the space of Riemannian metrics \\ on a compact manifold with boundary}}}

\author{Shota Hamanaka\thanks{supported in doctoral program in Chuo University, 2020.}}

\date{}

\maketitle

\begin{abstract}
In this paper, for a compact manifold $M$ with non-empty boundary $\partial M$,
we give a Koiso-type decomposition theorem,
as well as an Ebin-type slice theorem,
for the space of all Riemannian metrics on $M$
endowed with a fixed conformal class on $\partial M$.
As a corollary, we give a characterization of relative Einstein metrics.
\end{abstract}

\section{Introduction}

~~The study of the differential structure of the space $\mathscr{M}$ of all Riemannian metrics on a closed manifold is one of important studies in geometry.
In~\cite{ebin1970manifold}, Ebin particularly has proved a slice theorem for the pullback action of the diffeomorphism group on $\mathscr{M}$.
In \cite{koiso1978nondeformability}, Koiso has also extended it to an Inverse Limit Hilbert (ILH for brevity)-version.
Moreover, he has also studied the conformal action on $\mathscr{M}$, and consequently has proved the following decomposition theorem:

\begin{theo}[Koiso's decomposition theorem { \cite[Corollary 2.9]{koiso1979decomposition} }]
\label{theo1.1}
Let $M^{n}$ be a closed $n$-manifold $(n \ge 3)$, $\mathscr{M}$ the space of all Riemannian metrics on $M$ and $\mathrm{Diff}(M)$ the diffeomorphism group of $M$.
Set also 
$$C^{\infty}_{+}(M) := \bigl\{ f \in C^{\infty}(M) \bigm| f > 0~\mathrm{on}~M \bigr\},$$
$$\check{\mathfrak{S}} := \Biggl\{ g \in \mathscr{M} \Biggm| \mathrm{Vol}(M,g) = 1,~R_{g} = \mathrm{const},~\frac{R_{g}}{n-1} \notin \mathrm{Spec} (-\Delta_{g}) \Biggr\},$$
where $\mathrm{Vol}(M,g)$, $R_{g}$ and $\mathrm{Spec}(-\Delta_{g})$ 
denote respectively the volume of $(M,g)$, the scalar curvature of $g$ and the set of all non-zero eigenvalues of the (non-negative) Laplacian $-\Delta_{g}$ of $g$.
Note that these four spaces become naturally ILH-manifolds.
For any $g=f\bar{g}~(f \in C^{\infty}_{+},~\bar{g} \in \check{\mathfrak{S}})$ 
and any smooth deformation $\{ g(t) \}_{t \in (-\epsilon,\epsilon)}$ of $g$ for sufficiently small $\epsilon > 0$,
then there exist uniquely smooth deformations $\{ f(t) \}_{t \in (-\epsilon,\epsilon)} ( \subset C^{\infty}_{+}(M) )$ of $f$, $\{ \phi(t) \}_{t \in (-\epsilon,\epsilon)} ( \subset \mathrm{Diff}(M)$ ) of the identity $id_{M}$
and $\{ g(t) \}_{t \in (-\epsilon,\epsilon)} ( \subset \check{\mathfrak{S}}$ ) of $\bar{g}$ with $\delta_{g} (\bar{g}^{'}(0)) = 0$ such that
\[
g(t) = f(t) \phi(t)^{*} \bar{g}(t).
\]
Here, $\delta_{g} (\bar{g}^{'}(0))$ denotes the divergence $-~\nabla^{i} (\bar{g}^{'}(0))_{i}$ with respect to $g$.
\end{theo}

Note that the above decomposition can be replaced by 
\[
g(t) = (f(t) \circ \phi(t)) \phi(t)^{*}\bar{g}(t)~\mathrm{with}~\delta_{g} \big( f^{'}(0) \bar{g}(0) + f(0)\bar{g}^{'}(0) \big) = 0.
\]

This theorem has often played an important role in studying gemoetric structures related to several variational problems on a closed manifold.
Hence, extending this to on a manifold with boundary seems to be also important.

From now on, we throughout assume that $M$ is a compact connected oriented smooth $n$-manifold $(n \ge 3)$ with smooth boundary $\partial M$.
Let $\mathscr{M}$ be the space of all Riemannian metrics on $M.$
In order to obtain a corresponding Koiso-type decomposition theorem on $M$ with $\partial M$
to Theorem \ref{theo1.1} on a closed manifold, we need to fix a suitable boundary condition
for each metric $g$ on $M$.
From the variational view point of the Einstein-Hilbert functional,
a candidate of such boundary conditions may be the following:
For a fixed metric $g_{0}$ on $M$ with zero mean curvature $H_{g_{0}} = 0$
along $\partial M$, we fix the boundary condition for each $g$ on $M$ as 
$[g|_{\partial M}] = [g_{0} |_{\partial M}]$ and $H_{g} = 0$ along $\partial M$
(see Fact \ref{fact2.1} and \cite{akutagawa2},\cite{akutagawarelative}).
Here, $[g|_{\partial M}]$ denotes the conformal class of $g|_{\partial M}$.
However, one can notice that this boundary condition is not enough to get a 
Koiso-type decomposition theorem, even more an Ebin-type slice theorem.
Here, we will fix a slightly stronger boundary condition below,
which still has a naturality(see Fact \ref{fact2.1}(1)).

Fix a Riemannian metric $g_{0}$ on $M$ with $H_{g_{0}} = 0$ along $\partial M$
and set its conformal class $C := [g_{0}]$ on $M$.
$\nu_{g_{0}}$ denotes the outer unit normal vector field along $\partial M$ with respect to $g_{0}$.
When two metrics $g$ and $\tilde{g}$ on $M$ have the same 1-jets 
$j^{1}_{x} g = j^{1}_{x} \tilde{g}$ for all $x \in \partial M,$
we write it as $j^{1}_{\partial M} g = j^{1}_{\partial M} \tilde{g}.$
Set also 
\[
C^{\infty}_{+}(M)_{N} := \bigl\{ f \in C^{\infty}_{+}(M) \bigm| \nu_{g_{0}} (f)|_{\partial M} =0 \bigr\},
\]
\[
\mathscr{M}_{C_{0}} := \bigl\{ g \in \mathscr{M} \bigm| g = f g_{0}~\mathrm{on}~\partial M~\mathrm{for~some}~f \in C^{\infty}_{+}(M),~H_{g} = 0~\mathrm{on}~\partial M \bigr\},
\]
\[
\mathscr{M}_{C^{1}_{0}} := \bigl\{ g \in \mathscr{M} \bigm| j^{1}_{\partial M} g = j^{1}_{\partial M} (fg_{0})~\mathrm{for~some}~f \in C^{\infty}_{+}(M)_{N} \bigr\},
\]
\[
\mathfrak{S}_{C^{(1)}_{0}} := \bigl\{ g \in \mathscr{M}_{C^{(1)}_{0}} \bigm| \mathrm{Vol}(M,g) = 1,~R_{g} = \mathrm{const} \bigr\},
\]
\[
\check{\mathfrak{S}}_{C^{(1)}_{0}} := \Biggl\{ g \in \mathfrak{S}_{C^{(1)}_{0}} \Biggm| \frac{R_{g}}{n-1} \notin \mathrm{Spec} (-\Delta_{g} ; \mathrm{Neumann}) 
\Biggr\},
\]
\[
\mathrm{Diff}_{C_{0}} := \bigl\{ \phi \in \mathrm{Diff}(M) \bigm| \phi^{*} g_{0} = f g_{0}~\mathrm{on}~\partial M~\mathrm{for~some}~f \in C^{\infty}_{+}(M) \bigr\},
\]
where $\mathrm{Spec}(-\Delta_{g} ; \mathrm{Neumann})$ denote the set of all non-zero eigenvalues of $-\Delta_{g}$ with the Neumann boundary condition respectively.
Note that $H_{g} = 0$ along $\partial M$ for all $g \in \mathscr{M}_{C^{1}_{0}}$.
Our main result is the following theorem:

\begin{main}
For any $g = f\bar{g}~( f \in C^{\infty}_{+}(M)_{N},~\bar{g} \in \check{\mathfrak{S}}_{C^{1}_{0}}$ ) and any smooth deformation
$\{ g(t) \}_{t \in (-\epsilon,\epsilon)} ( \subset \mathscr{M}_{C^{1}_{0}} )$ of $g$ for sufficiently small $\epsilon > 0$ ,
there exist smooth deformations $\{ f(t) \}_{t \in (-\epsilon,\epsilon)} ( \subset C^{\infty}_{+}(M)_{N} )$ of $f,$
$\{ \phi(t) \}_{t \in (-\epsilon,\epsilon)} ( \subset \mathrm{Diff}_{C_{0}} )$ of $id_{M}$
and $\{ \bar{g}(t) \}_{t \in (-\epsilon,\epsilon)} ( \subset \check{\mathfrak{S}}_{C^{1}_{0}} )$ of $\bar{g}$
with $\delta_{g} (\bar{g}^{'}(0)) = 0$
such that 
\[
g(t) = f(t) \phi(t)^{*} \bar{g}(t).
\]
\end{main}

The rest of this paper is organized as follows.
In Section 2, we state a Slice theorem for a manifold with boundary with a fixed conformal class on the boundary and prove it.
In Section 3, we prepare some necessary lemmas for the proof of Main Theorem.
Finally, combining them with Slice theorem, we prove Main Theorem.
In Section 4, we give two applications of Main Theorem.

\subsection*{Acknowledgement}
~~I would like to thank my supervisor Kazuo Akutagawa for suggesting the initial direction for my study,
his good advice and support.

\section{Preliminaries and a slice theorem}

~~Let $M$ be a compact connected oriented smooth $n$-dimensional manifold with non-empty smooth boundary $\partial M$.
Fix a Riemannian metric $g_{0}$ with $H_{g_{0}} = 0$. Here, $H_{g_{0}}$ denotes the mean curvature of $\partial M$ with respect to $g_{0}$. And set $C := [g_{0}]$ its conformal class on $M$.
For a given positive definite symmetric (0,2)-type tensor field $T$ on $M$, we will write $T||_{\partial M} \in C||_{\partial M}$ when $T = f \cdot g_{0}~\mathrm{for~some}~f \in C^{\infty}_{+}(M)~\mathrm{on}~\partial M$.
Note that this condition equivalent to $\iota^{*} T = \iota^{*} (f g_{0})$ for some $f \in C^{\infty}_{+}(M)$, where $\iota~:\partial M \rightarrow M$ is the natural inclusion.
Moreover, we denote $T||^{1}_{\partial M} \in C_{0}||^{1}_{\partial M}$ when $j^{1}_{\partial M}T = j^{1}_{\partial M} (f \cdot g_{0})~\mathrm{for~some}~f \in C^{\infty}_{+}(M)_{N}$.
With this understood, we set    
$\mathscr{M}_{C} := \bigl\{ g \in \mathscr{M} \bigm| g||_{\partial M} \in C||_{\partial M} \bigr\}$. 
Note also that $\mathrm{Diff}_{C_{0}} := \bigl\{ \phi \in \mathrm{Diff}(M) \bigm| (\phi^{*} g_{0})||_{\partial M} \in C||_{\partial M} \bigr\}.$ and $\mathscr{M}_{C^{1}_{0}} = \bigl\{ g \in \mathscr{M} \bigm| g||^{1}_{\partial M} \in C||^{1}_{\partial M} \bigr\}
\subsetneq \mathscr{M}_{C_{0}} = \bigl\{ g \in \mathscr{M} \bigm| g||_{\partial M} \in C||_{\partial M},~H_{g} = 0~\mathrm{on}~\partial M \bigr\}.$

\begin{rema}
In the case that $\partial M= \emptyset$ (that is, $M$ is a closed manifold),
it is well known that a Riemannian metric on $M$ is Einstein if and only if it is a critical point of the normalized 
Einstein-Hilbert functional $\mathcal{E}$ on the space $\mathscr{M}:$
\[
\mathcal{E}~:~\mathscr{M} \rightarrow \mathbb{R},~~g \mapsto \mathcal{E}(g) := \frac{\bigintss_{M} R_{g}dv_{g}}{\mathrm{Vol}_{g}(M)^{\frac{n-2}{n}}},
\]
where $R_{g},~dv_{g},~\mathrm{Vol}_{g}(M)$ denote respectively the scalar curvature, the volume measure of $g$ and the volume of $(M,g)$.
However, if we consider the analogue of the case of $\mathcal{E}$ on compact $n$-manifold $M$ with non-empty boundary,
then the set of critical points of $\mathcal{E}$ on the space $\mathscr{M}$ is empty (see Fact \ref{fact2.1} below). 
Hence, in this case, we need to fix a suitable boundary condition for all metrics, and then $\mathcal{E}$ must be restricted to a subspace of $\mathscr{M}$.
\vspace{1ex}

When $\partial M \neq \emptyset,$ set the several subspaces of $\mathscr{M}$ below:
\[
\mathscr{M}_{C|_{\partial}} := \bigl\{ g \in \mathscr{M} \bigm| [g|_{\partial M}] = C|_{\partial M} \bigr\},
\]
\[
\mathscr{M}_{C_{\mathrm{const}}|_{\partial}} := \bigl\{ g \in \mathscr{M}_{C|_{\partial}} \bigm| \exists c \in \mathbb{R}~\mathrm{s.t.}~H_{g} = c~\mathrm{on}~\partial M \bigr\},
\]
\[
\mathscr{M}_{0} := \bigl\{ g \in \mathscr{M} \bigm| H_{g} = 0~\mathrm{on}~\partial M \bigr\},
\]
\[
\mathscr{M}_{C_{0}|_{\partial}} := \mathscr{M}_{C|_{\partial}} \cap \mathscr{M}_{0} = \bigl\{ g \in \mathscr{M}_{C|_{\partial}} \bigm| H_{g} = 0~\mathrm{on}~\partial M \bigr\}.
\]

A metric $g \in \mathscr{M}$ is called a \textit{relative} metric if $g \in \mathscr{M}_{0}.$
By Fact \ref{fact2.1} below, it is reasonable to restrict the functional $\mathcal{E}$ to the subspace
$\mathscr{M}_{C_{0}}$ as well as  $\mathscr{M}_{C_{0}|_{\partial}}$ and $\mathscr{M}_{0}.$ 

\begin{fact}[{\cite[Proposition~2.1]{akutagawa2}},~{\cite[Remark~1,~Theorem~1.1]{akutagawarelative}}]
\label{fact2.1}
Let $M,~\mathcal{E},$ and $\mathscr{M}$ be the same as the above. 
Then the following holds:

(1) $g \in \mathrm{Crit}(\mathcal{E}|_{\mathscr{M}_{C^{(1)}_{0}}})$ if and only if $g$ is an Einstein metric with $H_{g} = 0$ 
(namely, a relative Einstein metric) and $g||^{(1)}_{\partial M} \in C||^{(1)}_{\partial M}.$

(2) $g \in \mathrm{Crit}(\mathcal{E}|_{\mathscr{M}_{C_{0}|_{\partial}}})$ if and only if $g$ is a relative Einstein metric and $[g|_{\partial M}] = C|_{\partial M}$.

(3) $g \in \mathrm{Crit}(\mathcal{E}|_{\mathscr{M}_{0}})$ if and only if $g$ is an Einstein metric with totally geodesic boundary.

(4) $\mathrm{Crit}(\mathcal{E}) = \emptyset,~~\mathrm{Crit}(\mathcal{E}|_{\mathscr{M}_{C}}) = \emptyset,~~\mathrm{Crit}(\mathcal{E}_{\mathscr{M}_{C|_{\partial}}}) = \emptyset,~~\mathrm{Crit}(\mathcal{E}_{\mathscr{M}_{C_{\mathrm{const}}|_{\partial}}}) = \emptyset.$
\noindent
Here, for instance, $\mathrm{Crit}(\mathcal{E})$ and $\mathrm{Crit}(\mathcal{E}|_{\mathscr{M}_{C}})$ denote respectively the set of all critical metrics of $\mathcal{E}$
and the set of those of its restriction to $\mathscr{M}_{C}$.
\end{fact}

\end{rema}

 From now on, we will consentrate on the spaces $\mathscr{M}_{C}$ and $\mathscr{M}_{C_{0}}.$
For a smooth fibre bundle $F$ , we denote by  $H^{s}(F)$ the space of all $W^{s,2}$ -sections.
(Note that $L^{2}$ -norm does not depend on the choice of Riemannian metric, hence, we fix a reference metric to define these function spaces.)
The Sobolev embedding theorem states that $H^{s}(F) \hookrightarrow C^{k}(F)$ is continuous if $s > n / 2 + k$, see for instance \cite{aubin2013some}.
By the Sobolev embedding, if $s > n / 2$, $H^{s}(M \times M)$ (the set of all $H^{s}$-maps from $M$ to itself) is a Hilbert manifold.
Pick $s > 4 + \frac{n}{2}$ and let $C^{s} \mathrm{Diff} := \bigl\{ \eta \in C^{s}(M \times M) \bigm| \eta^{-1} \in C^{s}(M \times M) \bigr\}$
and let $\mathrm{Diff}^{s} := H^{s}(M \times M) \cap C^{1} \mathrm{Diff}$. 
From the Sobolev embedding, $\mathrm{Diff}^{s}$ is open in $H^{s}(M \times M)$ and hence it is also a Hilbert manifold.
Let $\mathrm{Diff}^{s}_{C_{0}} := \bigl\{ \eta \in \mathrm{Diff}^{s} \bigm| (\eta^{*} g_{0})||_{\partial M} \in C||_{\partial M} \bigr\} 
= \bigl\{ \eta \in \mathrm{Diff}^{s} \bigm| \eta^{*} g_{0} = f g_{0}~\mathrm{on}~\partial M~\mathrm{for~some}~f \in H^{s-3/2}(C^{\infty}_{+}(M)) \bigr\}$.
Then $\mathrm{Diff}^{s}_{C_{0}}$ is a Hilbert submanifold of $\mathrm{Diff}^{s - 2}$.
We denote by $id_{M} \in \mathrm{Diff}^{s}_{C_{0}}  ( \subset \mathrm{Diff}^{s} )$ the identity map.

We set $\mathscr{M}^{s} := H^{s}(S^{2}T^{*}M) \cap C^{0}\mathscr{M}$, where $S^{2}T^{*}M$ and $C^{0}\mathscr{M}$ the tensor field consisting of all symmetric (0,2)-tensors on $M$
and the set of all $C^{0}$ metrics respectively. 
Then $\mathscr{M}^{s}$ is a Hilbert manifold modeled on $H^{2}(S^{2}T^{*}M)$ (by the Sobolev embedding).
And we define a closed Hilbert submanifold of $\mathscr{M}^{s - 1}$ as $\mathscr{M}^{s}_{C} := \bigl\{ g \in \mathscr{M}^{s} \bigm| g||_{\partial M} \in C||_{\partial M} \bigr\}
= \bigl\{ g \in \mathscr{M}^{s} \bigm| g = f g_{0}~\mathrm{for~some}~f \in H^{s -1/2}(C^{\infty}_{+}(M))~\mathrm{on}~\partial M \bigr\}$.
Additionally, we set $\mathscr{M}^{s}_{C_{0}} := \bigl\{ g \in \mathscr{M}^{s} \bigm| g = f g_{0}~\mathrm{on}~\partial M~\mathrm{for~some}~f \in H^{s -1/2}(C^{\infty}_{+}(M))~
\mathrm{and}~H_{g} = 0~\mathrm{on}~\partial M \bigr\}$
and
$\mathscr{M}^{s}_{C^{1}_{0}} := \bigl\{ g \in \mathscr{M}^{s} \bigm| g ||^{1}_{\partial M} \in C ||^{1}_{\partial M} \bigr\}
= \bigl\{ g \in \mathscr{M}^{s} \bigm| j^{1}_{\partial M}g = j^{1}_{\partial M} (f g_{0})~\mathrm{for~some}~f \in H^{s -1/2}(C^{\infty}_{+}(M)_{N}) \bigr\}$.
Then $\mathscr{M}^{s}_{C_{0}}$ and $\mathscr{M}_{C^{1}_{0}}$ are Hilbert submanifolds of $\mathscr{M}^{s-2}$. See \cite{ebin1970manifold}, \cite{palais1968foundations} for more detail about these spaces. 
Note that, for $\eta \in \mathrm{Diff}^{s}_{C_{0}}$ and $g \in \mathscr{M}_{C^{(1)}_{0}},~\eta^{*} g \in \mathscr{M}_{C^{(1)}_{0}}.$
Moreover, for $g \in \mathscr{M}$, we denote $I_{g}$ be the isotropy subgroup of $g$ in $\mathrm{Diff}^{s}_{C_{0}}$
, that is, 
\[
I_{g} := \bigl\{ \eta \in \mathrm{Diff}^{s}_{C_{0}} \bigm| \eta^{*}g = g \bigr\}.
\]

In this section, we shall prove the following theorem:

\begin{theo}[Slice theorem for manifold with boundary]
\label{theo2.1}

Let $s > \frac{n}{2} + 4$ and
\[
A~:~\mathrm{Diff}^{s+1}_{C_{0}} \times \mathscr{M}^{s}_{C} \longrightarrow \mathscr{M}^{s}_{C}
\]
be an usual action by pullback.
Then for each $\gamma \in \mathscr{M}_{C}$ there exsits a submanifold $\mathcal{S} \subset \mathscr{M}^{s}_{C}$
containing $\gamma$ ,which is diffeomorphic to a ball in a separable real Hilbert space, such that

(1) $\eta \in I_{\gamma} \Rightarrow A(\eta,\mathcal{S}) = \mathcal{S},$

(2) $\eta \in \mathrm{Diff}^{s+1}_{C_{0}}~,~A(\eta,\mathcal{S}) \cap \mathcal{S} \neq \emptyset \Rightarrow \eta \in I_{\gamma}$ and

(3)There exists a local section:
\[
\chi~:~\bigl( \mathrm{Diff}^{s+1}_{C_{0}} / I_{\gamma} \supset \bigr)U \longrightarrow \mathrm{Diff}^{s+1}_{C_{0}}
\]
defined in a neighborhood $U$ of the identity coset such that if  
\[
F~:~U \times \mathcal{S} \longrightarrow \mathscr{M}^{s}_{C}~;~(u,t) \mapsto A \bigl( \chi(u),t \bigr),
\]
then $F$ is a homeomorphism onto a neighborhood of $\gamma.$
Moreover, the same statement holds when we replace $\mathscr{M}^{s}_{C}$ by $\mathscr{M}^{s}_{C^{(1)}_{0}}$.
\end{theo}

First, we get the following lemma:
\begin{lemm}
\label{lemm2.9}
Under the identification 
$T_{id_{M}}\mathrm{Diff}^{s + 1} \cong H^{s + 1}(TM),$
\[
\begin{split}
(1)~T_{id_{M}} \mathrm{Diff}^{s+1}_{C_{0}} 
= \bigl\{ X &\in H^{s+1}(TM) \bigm| \exists \rho \in H^{s - 1/2}(M), \\ 
&^{g_{0}}\nabla_{i}X_{j} + ^{g_{0}}\nabla_{j}X_{i} = \rho g_{0}, 
~g_{0}(X,\nu_{g_{0}}) = 0~\mathrm{on}~\partial M \bigr\},
\end{split}
\]
where $^{g_{0}}\nabla$ and $X_{i}$ are respectively the Levi-Civita connection with respect to $g_{0}$ and the $i$-th component $g_{ij} X^{j}$ of $X = ( X^{j} )$ in terms of some local coordinates.
For $g \in \mathscr{M}^{s}_{C}$, 
\[
(2)~~~~~T_{g}\mathscr{M}^{s}_{C} = \bigl\{ h \in H^{s}(S^{2}T^{*}M) \bigm| \exists \rho \in H^{s - 1/2}(M), h = \rho g~\mathrm{on}~\partial M \bigr\}.
\]
Here, $T_{id_{M}} \mathrm{Diff}^{s+1}_{C_{0}}$ and $T_{g}\mathscr{M}^{s}_{C}$ represent the tangent spaces respectively.
Moreover, for $g \in \mathscr{M}^{s}_{C_{0}}$, we also have
\[
\begin{split}
(3)~~~~T_{g}\mathscr{M}^{s}_{C_{0}} = \bigl\{ h \in H^{s}(S^{2}T^{*}M) \bigm| &\exists \rho \in H^{s - 1/2}(M), h = \rho g~\mathrm{on}~\partial M, \\
& D_{g} H (h) = 0~\mathrm{on}~\partial M \bigr\},
\end{split}
\]
where $D_{g} H$ denotes the derivative of the mean curvatrue function $H$ at $g.$
And, for $g \in \mathscr{M}^{s}_{C^{1}_{0}},$
\[
(4)~~~~T_{g}\mathscr{M}^{s}_{C^{1}_{0}} = \bigl\{ h \in H^{s}(S^{2}T^{*}M) \bigm| \exists \rho \in H^{s - 1/2}(M), j^{1}_{\partial M} h = j^{1}_{\partial M}(\rho g) \bigr\}.
\]
\end{lemm}
 
\begin{proof}
From the definition, (2), (3) and (4) are obvious.
For proving (1), we note that the derivative of a diffeomorphism-action via pull-back on metrics $g$ at $id_{M}$ coincides with the Lie derivative of $g$
under the above identification (see for instance Lemma~6.2 in \cite{ebin1970manifold}).
Then the first condition comes from the conformal condition on $\partial M$ and the second from the fact that any diffeomorphisms map $\partial M$ to itself.
And (3) follows in the same way.

Here, we also note that $\partial M$ may have some connected components. (Note that the number of components is finite since $M$ is compact.)
In fact, we assume that $\partial M = \coprod^{k}_{i = 1} \Sigma_{i}$, where $\Sigma_{i}$ is a connected component of $\partial M$ and $k \in \mathbb{Z}_{\ge 1}.$ 
Then, under the above identification,
$\eta_{t}(\Sigma_{i}) = \Sigma_{i}~\mathrm{for~all}~t \in (-\epsilon, \epsilon)$ for sufficiently small $\epsilon > 0$
as explained below.
Here, $\eta_{t} \in T_{id_{M}}\mathrm{Diff}^{s + 1}$ is the corresponding curve to a tangent vector.
Since $M$ is compact manifold, we can take some open neighborhoods of each $\Sigma_{i}$, $U_{i}$ such that 
$U_{i} \cap U_{j} = \emptyset$ for all $i \neq j$.
Consider $W(\partial M,U) := \bigl\{ f \in C^{0}(M,M) \bigm| f(\Sigma_{i}) \subset U_{i} \bigr\}$, 
then this is an open subset of $C^{0}(M,M)$ with respect to the compact-open topology.
Hence, this is an open neighborhood of $id_{M}$ in $C^{0}(M,M)$.
Since $\eta_{0} = id_{M}$ and $t \mapsto \eta_{t}$ is continuous, $\eta_{t} \in W(\partial M, U)$ for all $t$ with $|t| << 1$.
In paticular, $\eta_{t}(\Sigma_{i}) = \Sigma_{i}$ for all $i$ and $t$ with $|t| << 1$.
\end{proof}

From {\cite[Section 9]{palais1968foundations}}, $H^{s}(TM)$ is linearly isomorphic to a closed subspace of finite direct sum of $H^{s}(D_{n},\mathbb{R})$,
where $D_{n}$ is a closed $n$-dimensional disc. Therefore we obtain the following lemma in exactly the same way as in {\cite[Section 3]{ebin1970manifold}}.

\begin{lemm}[{\cite[Section 3]{ebin1970manifold}}]
$\mathrm{Diff}_{C_{0}}$ is a topological group.
Furthermore, for all  
$\sigma \in \mathrm{Diff}_{C_{0}} := \bigcap_{s \ge n/2 + 5} \mathrm{Diff}^{s}_{C_{0}}$,
the left(right) action  
$L_{\eta}(R_{\eta})~:~\mathrm{Diff}^{s}_{C_{0}} \rightarrow \mathrm{Diff}^{s}_{C_{0}}$ is smooth.
\end{lemm}

\begin{rema}
It is well known that any $C^{1}$-diffeomorphism, which is an isometry of a smooth metric, is smooth~(see \cite{ebin1970manifold},~\cite{kobayashi1963foundations},~\cite{myers1939group}).
Therefore, from the Sobolev embedding, $I_{\gamma} \subset \mathrm{Diff}_{C_{0}}$ and $I_{\gamma}$ is the same for any such that $s > \frac{n}{2} + 1$.
\end{rema}

The following lemma follows from Section 5 in \cite{ebin1970manifold} and the fact that $\mathrm{Diff}^{s}_{C_{0}}$ is a submanifold of $\mathrm{Diff}^{s-2}$.
\begin{lemm}[{\cite[Section 5]{ebin1970manifold}}]
\label{lemm2.4}
(1) The natural inclusion
\[
i~:~I_{\gamma} \longrightarrow \mathrm{Diff}^{s}_{C_{0}}
\]
is smooth.

(2) $i$ is embedding, that is
for all $g \in I_{\gamma}$ , its derivative $D_{g}i$ is injective and its image is closed in 
$T_{g} \mathrm{Diff}^{s}_{C_{0}}.$

(3) The composition map:
\[
c~:~I_{\gamma} \times \mathrm{Diff}^{s}_{C_{0}} \longrightarrow \mathrm{Diff}^{s}_{C_{0}};~(g,\eta) \longmapsto g \circ \eta
\]
is smooth.

(4) Let $S := \bigcup_{\eta \in \mathrm{Diff}^{s}_{C}}T_{id_{M}}R_{\eta}(\mathscr{I})$
, then $S$ is a $C^{\infty}$ involutive subbundle of $T \mathrm{Diff}^{s}_{C_{0}},$
that is,
\[
X,Y \in S \Rightarrow [X,Y] \in S,
\]
where $\mathscr{I}$ is the Lie algebra of $I_{\gamma}$.
\end{lemm}

Using lemma~\ref{lemm2.4} and the Frobenius's theorem(see~{\cite[Chapter~6,~Theorem~2]{lang}}), 
we obtain a Banach manifold structure on $\mathrm{Diff}^{s}_{C_{0}} / I_{\gamma}$ (see~{\cite[Proposition~5.8,~5.9]{ebin1970manifold}}).
And we can show the exsitence of a local section by using this Banach-coordinate-charts:

\begin{lemm}[{\cite[Proposition~5.10]{ebin1970manifold}}]
For any $I_{\gamma}\eta \in \mathrm{Diff}^{s}_{C_{0}} / I_{\gamma},$ 
there exists a local section 
$\pi:~\mathrm{Diff}^{s}_{C_{0}} \longrightarrow \mathrm{Diff}^{s}_{C_{0}} / I_{\gamma}$ 
defined on a neighborhood $I_{\gamma}\eta$.
\end{lemm}

\begin{proof}
Using the above Banach-coordinate-charts, we can construct a local section exactly same as Proposition 5.10 in \cite{ebin1970manifold}.
\end{proof}

For any $\gamma \in \mathscr{M}_{C}$, $\sigma \in \mathscr{M}_{C^{(1)}_{0}}$, we define
\[
\psi_{\gamma}:\mathrm{Diff}^{s+1}_{C_{0}} \longrightarrow \mathscr{M}^{s}_{C}~;~\psi_{\gamma}(\eta) := \eta^{*}\gamma
\]
and
\[
\psi^{0}_{\sigma}:\mathrm{Diff}^{s+1}_{C_{0}} \longrightarrow \mathscr{M}^{s}_{C^{(1)}_{0}}~;~\psi_{\sigma}(\eta) := \eta^{*}\sigma.
\]
Then, by the definition of $I_{\gamma}$, it naturally induces
$\phi_{\gamma}:~\mathrm{Diff}^{s}_{C_{0}}/I_{\gamma} \longrightarrow \mathscr{M}^{s}_{C}$ and $\phi^{0}_{\sigma}:~\mathrm{Diff}^{s}_{C_{0}}/I_{\sigma} \longrightarrow \mathscr{M}^{s}_{C^{(1)}_{0}}$.
Moreover, from the existence of a local section and the definition, $\phi_{\gamma},~\phi^{0}_{\sigma}$ are smooth injective maps.

Here, we will show that $\phi_{\gamma}:~\mathrm{Diff}^{s+1}_{C_{0}} / I_{\gamma} \rightarrow \mathscr{M}^{s}_{C}$ 
and $\phi^{0}_{\sigma}:~\mathrm{Diff}^{s+1}_{C_{0}} / I_{\sigma} \rightarrow \mathscr{M}^{s}_{C^{(1)}_{0}}$
are immersions (i.e., its derivation is injective and has closed image).

\begin{rema}
We can similarly define $\phi_{\gamma}:~\mathrm{Diff}^{s} / I_{\gamma} \rightarrow \mathscr{M}^{s}$,~but it is not an immersion in general.
\end{rema}

Next, we will show that the image of $D_{\eta}\psi_{\gamma}$ is closed in $T_{\psi_{\gamma}(\eta)} \mathscr{M}^{s}_{C}$ and 
imege of $D_{\eta} (\psi^{0}_{\sigma})$ is closed in $T_{\psi^{0}_{\sigma}(\eta)} \mathscr{M}^{s}_{C^{(1)}_{0}}$.

As mentioned in the proof of lemma \ref{lemm2.9}, under the identification as in lemma \ref{lemm2.9},
$D_{id_{M}}\psi_{\gamma} (X) = \mathscr{L}_{X}(\gamma)~(X \in T_{id_{M}} \mathrm{Diff}_{C_{0}}^{s+1})$, 
and $D_{id_{M}} \psi^{0}_{\sigma} (X) = \mathscr{L}_{X}(\sigma),~(X \in T_{id_{M}} \mathrm{Diff}_{C_{0}}^{s+1})$.
For simplicity, we put the Lie derivative with respect to $\gamma$ as $\alpha$, that is $\alpha(\bullet) := \mathscr{L}_{\bullet}(\gamma).$

For $X \in H^{s + 1}(M),~Y_{i} \in H^{s}(M),
~Z_{ij} \in H^{s - 1}(M)~( i,j = 1, \dots, n),$ we set
\[
A^{s+1}_{(X,Y,Z)} := \bigl\{ u \in T_{id_{M}} \mathrm{Diff}^{s+1}_{C_{0}} \bigm| u = X ,
\nabla_{i} u = Y_{i} , \nabla^{2}_{ij} u = Z_{ij}~\mathrm{on}~\partial M \bigr\},
\]
\[
\begin{split}
B^{s}_{(X,Y,Z)}
:= \bigl\{ \eta \in T_{\gamma}\mathscr{M}^{s}_{C} \bigm| \eta(\bullet , *)  
&= \gamma \bigl( Y(\bullet) , * \bigr) + \gamma \bigl( \bullet , Y(*) \bigr), \\
&-~_{\gamma}Tr_{1,2}(\nabla \eta) = 4 \bigl( \delta \delta^{*} u \bigr)~\mathrm{on}~\partial M \bigr\},
\end{split}
\]
where
$_{\gamma}Tr_{1,2}$, $\delta$ and $\delta^{*}$ denote respectively the (1,2)-contraction with respect to $\gamma,$
the divergence operator of $\gamma$ : $\delta = -_{\gamma}Tr_{1,2} \circ \nabla$
and its formal adjoint : $\delta^{*} = Sym \circ \nabla$.
And `` $u$ " in the definition of $B^{s}_{(X,Y,Z)}$ is an element of $T_{id_{M}} \mathrm{Diff}^{s+1}_{C_{0}}$ such that
$\nabla_{i} u |_{\partial M} = Y_{i}~\mathrm{and}~\nabla^{2}_{ij} u |_{\partial M} = Z_{ij}.$
Thus, $\bigl( \delta \delta^{*} u \bigr) |_{\partial M}$ is determined by them.
And these are Hilbert spaces.

Since the composition map $\alpha^{*} \alpha$ is elliptic, we obtain 
the following boundary estimate (see {\cite[Theorem 6.6]{gilbarg2013elliptic}}):

There exists a positive constant $C$ such that for any $u \in A^{s + 1}_{(X,Y,Z)}$, 
\[
||u||_{k} \le C \bigl( \bigl|\bigl| (\alpha^{*} \alpha)u \bigr|\bigr|_{k-2} + ||u||_{k-2} + \bigl|\bigl| \Tilde{X} \bigr|\bigr|_{k} \bigr) 
\]
where $\Tilde{X}$ is an extension to $M$ of $X.$

For $u \in T_{id.} \mathrm{Diff}^{s+1}_{C_{0}}$ and $h \in T_{\gamma}\mathscr{M}^{s}_{C}$ , from the Green's formula,
\[
(\alpha u, h) = (u,\alpha^{*} h) + 2\bigl( u^{\flat},h(\nu,\bullet) \bigr)_{\partial M},
\]
where 
$(,)~\mathrm{and}~(,)_{\partial M}$ are $L^{2}$-inner product with respect to $\gamma~\mathrm{and}~\gamma |_{\partial M}$ .
And $\nu$ is the outer unit normal vector along $\partial M$ with respect to $\gamma.$

Since $h \in T_{id.} \mathrm{Diff}^{s+1}_{C_{0}}$, from lemma \ref{lemm2.9}, the second term vanishes. Thus we get  
\[
(\alpha u, h) = (u,\alpha^{*} h).
\]
Therefore 
for all $u,v \in T_{id.} \mathrm{Diff}^{s+1}_{C_{0}}$,
\begin{equation}
\bigl( (\alpha^{*} \alpha) u, v \bigr) = \bigl( u,(\alpha^{*} \alpha) v \bigr).  \label{eq:orth}
\end{equation}

From the closed range theorem({\cite[Lemma 5.10]{gilbarg2013elliptic}}), Proposition~6.8 and 6.9 in \cite{ebin1970manifold}, we get the following:

\begin{lemm}
\label{lemm3.15}

$\alpha(A^{s+1}_{(X,Y,Z)})$ is a closed subspace of $B^{s}_{(X,Y,Z)}$ and there is an orthogonal decomposition:
\[
B^{s}_{(X,Y,Z)} = \mathrm{Im} \bigl( \alpha |_{A^{s+1}_{(X,Y,Z)}} \bigr) \oplus \mathrm{Ker} \bigl( \alpha^{*} |_{B^{s}_{(X,Y,Z)}} \bigr).
\]
\end{lemm}

From lemma \ref{lemm3.15}, we get that
\[
T_{\gamma}\mathscr{M}^{s}_{C} = \mathrm{Im} \bigl( \alpha |_{T_{id.} \mathrm{Diff}^{s+1}_{C_{0}}} \bigr) + \mathrm{Ker} \bigl( \alpha^{*} |_{T_{\gamma}\mathscr{M}^{s}_{C}} \bigr).
\]
On the other hand, from the equation~(\ref{eq:orth}),
\[
\mathrm{Im} \bigl( \alpha |_{T_{id.} \mathrm{Diff}^{s+1}_{C_{0}}} \bigr) \perp \mathrm{Ker} \bigl( \alpha^{*} |_{T_{\gamma}\mathscr{M}^{s}_{C}} \bigr).
\]
Hence we get an orthogonal decomposition
\[
T_{\gamma}\mathscr{M}^{s}_{C} = \mathrm{Im} \bigl( \alpha |_{T_{id.} \mathrm{Diff}^{s+1}_{C_{0}}} \bigr) \oplus \mathrm{Ker} \bigl( \alpha^{*} |_{T_{\gamma}\mathscr{M}^{s}_{C}} \bigr).
\]
In paticular, $\mathrm{Im}~D_{id.}\psi$ is closed subspace of $T_{\gamma}\mathscr{M}^{s}_{C}$.
Therefore, for each $\eta \in \mathrm{Diff}^{s+1}_{C_{0}}$,
$\mathrm{Im}~D_{\eta}\psi_{\gamma}$ is a closed subspace of $T_{\psi_{\gamma}(\eta)} \mathscr{M}^{s}_{C}$.

Similarly, (however we consider up to upper order boundary datum,) we also get the decomposition
\[
T_{\gamma}\mathscr{M}^{s}_{C^{(1)}_{0}} = \mathrm{Im} \bigl( \alpha |_{T_{id.} \mathrm{Diff}^{s+1}_{C_{0}}} \bigr) \oplus \mathrm{Ker} \bigl( \alpha^{*} |_{T_{\gamma}\mathscr{M}^{s}_{C^{(1)}_{0}}} \bigr).
\]
and that $\mathrm{Im}~D_{\eta}\psi^{0}_{\sigma}$ is a closed subspace of $T_{\psi^{0}_{\sigma}(\eta)}\mathscr{M}^{s}_{C^{(1)}_{0}}$ for all $\eta \in \mathrm{Diff}^{s+1}_{C_{0}}$.

Moreover, we can show that $D_{[\eta]}\phi_{\gamma}$ and $D_{[\eta]} \phi^{0}_{\sigma}$ are injective in the same way as Proposition 6.11 in \cite{ebin1970manifold}.
Consequently, we obtain the following:

\begin{lemm}
\label{lemm2.16}

\[
\phi_{\gamma}~:~\mathrm{Diff}^{s+1}_{C_{0}} / I_{\gamma} \longrightarrow \mathscr{M}^{s}_{C}~\mathrm{and}~
\phi^{0}_{\sigma}~:~\mathrm{Diff}^{s+1}_{C_{0}} / I_{\gamma} \longrightarrow \mathscr{M}^{s}_{C^{(1)}_{0}}
\]
are smooth injective immersions.
\end{lemm}

Moreover, we can prove the following in the same way as Proposition 6.13 in \cite{ebin1970manifold}:

\begin{lemm}
\label{lemm2.17}
Let $s > \frac{n}{2} + 4$ , then
\[
\phi_{\gamma}~:~\mathrm{Diff}^{s+1}_{C_{0}} / I_{\gamma} \longrightarrow \mathscr{M}^{s}_{C}
\]
is a homeomorphism mapping $\mathrm{Diff}^{s+1}_{C_{0}} / I_{\gamma}$ onto a closed subspace of $\mathscr{M}^{s}_{C}$ .
Therefore, in paticular, $\phi_{\gamma}$ is an embedding.
The same statement also holds when we replace respectively $\mathscr{M}^{s}_{C}$ and $\phi$ by $\mathscr{M}^{s}_{C^{(1)}_{0}}$ and $\phi^{0}$.
\end{lemm}

\begin{rema}
The connectedness of $M$ was used in the proof of lemma~\ref{lemm2.17}
(see the proof of Proposition 6.13 in \cite{ebin1970manifold}).
\end{rema}

\begin{proof}[Proof of the Theorem~\ref{theo2.1}]

The proof is same as in \cite{ebin1970manifold}.
Because we can show in the same way for $\mathscr{M}_{C^{(1)}_{0}}$, we only descrie about $\mathscr{M}_{C}$.

For $\gamma \in \mathscr{M}_{C}$ , we set $O^{s}_{\gamma} := \phi_{\gamma} \bigl( \mathrm{Diff}^{s+1}_{C_{0}} / I_{\gamma} \bigr)$ the orbit of the action $A$ through $\gamma$
(where $\phi_{\gamma}$ is in the lemma~\ref{lemm2.17}).
From Lemma~\ref{lemm2.17}, this is a closed submanifold of $\mathscr{M}^{s}_{C}.$
Moreover, we define its normal vector bundle:
$$\nu := \bigl\{ V \in T\mathscr{M}^{s}_{C}|_{O^{s}_{\gamma}} \bigm| (W,V)_{\gamma} = 0~,~\forall W \in TO^{s}_{\gamma} \bigr\},$$ 
where  
$(*,*)_{\gamma} := \bigintss_{M} \langle *,* \rangle_{\gamma}dv_{\gamma}.$

Our first step is constructing the normal bundle $\nu$ of $O^{s}_{\gamma}$ in $\mathscr{M}^{s}_{C}.$
As stated in \cite{ebin1970manifold}, this Riemannian metric is strong on $H^{0},$ but is not on $H^{s}~(s \ge 1).$
Thus we do not know automatically that $\nu$ is a $C^{\infty}$ subbundle of $T\mathscr{M}^{s}_{C} |_{O^{s}_{\gamma}}.$

To show this, we shall find a $C^{\infty}$ surjective vector-bundle-map: 
\[
P~:~T\mathscr{M}^{s}_{C} |_{O^{s}_{\gamma}} \longrightarrow TO^{s}_{\gamma}
\]
such that $\mathrm{Ker}~P = \nu$
(see {\cite[Chapter3,~Section3]{lang}}).

Since, from the proof of the lemma~\ref{lemm2.16},
\[
T_{\gamma}\mathscr{M}^{s}_{C} = \mathrm{Im}~\alpha \oplus \mathrm{Ker}~\alpha^{*}.
\]
Hence, from the definition of $O^{s}_{\gamma},$
\[
\mathrm{Im}~\alpha = T_{\gamma}O^{s}_{\gamma}.
\]
Thus
\[
\nu_{\gamma}(O^{s}_{\gamma}) = \mathrm{Ker} \bigl( \alpha^{*} |_{T_{\gamma}\mathscr{M}^{s}_{C}} \bigr).
\]

Moreover, since the weak Riemannian metric  $(~,~)_{\gamma}$ is invariant under the action of $\mathrm{Diff}^{s+1}_{C}$ ({\cite[Section~4]{ebin1970manifold}}),
$\nu_{\eta^{*}\gamma}(O^{s}_{\gamma}) = \eta^{*} \bigl( \mathrm{Ker}~\alpha^{*} \bigr).$
On the orbit of $\gamma,$ we define 
\[
P := \alpha \circ (\alpha^{*} \circ \alpha)^{-1} \circ \alpha^{*}~:~T_{\gamma}\mathscr{M}^{s}_{C} |_{O^{s}_{\gamma}} \longrightarrow T_{\gamma}O^{s}_{\gamma}
\]
Thus, as in the same way in {\cite[Theorem~7.1]{ebin1970manifold}}, we can show that this $P$ satisfy the above properties.

Next, we shall construct the slice $\mathcal{S}$ of $\gamma.$
To do this, we consider the exponential map of $(~,~)_{\gamma},$ $\exp:~T\mathscr{M}^{s}_{C} \longrightarrow \mathscr{M}^{s}_{C}.$
Thus we know the following fact:

\begin{fact}[{\cite[Section~4]{ebin1970manifold}}]
This is a smooth map and  $exp |_{\nu}:~\nu \longrightarrow \mathscr{M}^{s}_{C}$ is a diffeomorphism mapping a neighborhood of the zero section of $\nu$
to a neighborhood of $O^{s}_{\gamma}$ in $\mathscr{M}^{s}_{C}.$
Moreover, since $A$ is continuous and $exp$ and the action of $\eta$ are commutative, 
there are a neighborhood $U$ of $\gamma$ in $O^{s}_{\gamma}$ and a neighborhood $V$ of $0$ in $\nu_{\gamma}$ such that 
\[
\nu \supset W := \bigl\{ \eta^{*}(v) \bigm| v \in V~,~\eta \in \chi(U) \bigr\}.
\]
Then $exp |_{W}$ is  a diffeomorphism mapping $W$ onto a neighborhood of $\gamma.$
Moreover, if necessary, we shall take $U$ and $V$ small enough so that $\exp (W) \cap O^{s}_{\gamma} = U.$
\end{fact}

Consider the strong inner product $(~,~)_{\gamma}^{s}$ on $H^{s}(S^{2}T^{*}),$ defined as at the end of Section~4 in \cite{ebin1970manifold}.
Now let $\rho_{s}$ be the metric defined on $\mathscr{M}^{s}_{C}$ by $(~,~)^{s}_{\gamma}.$
Let $B^{r}_{\gamma}$ be the open ball about $\gamma$ of radius $r$ with respect to $\rho_{s}.$
Then, for some positive $\delta,$ $\exp (W) \supset B^{2 \delta}_{\gamma}.$
Pick $U_{1} \subset U,~\epsilon_{1} < \epsilon~(\leadsto V_{1} \subset V)$ so that if
$W_{1} := \bigl\{ \eta^{*}(v) \bigm| v \in V_{1}~,~\eta \in \chi \bigl( U_{1} \bigr) \bigr\},$
then $\exp (W_{1}) \subset B^{\delta}_{\gamma}.$

Then we set
\[
\mathcal{S} := \exp (V_{1})
\]
and this $\mathcal{S}$ has the three properties of a slice
(These are checked in the same way in {\cite[Section~7~and~the~proof~of~Theorem~7.1]{ebin1970manifold}}).

\end{proof}

\section{Main Results}

~~Before starting the proof of Main theorem, we shall line up some basic definitions below:

\begin{defi}[\cite{omori1970group}]

(1)~A topological space $E$ is called {\it ILH-space}
if $E$ is an inverse limit of Hilbert spaces $\{ E_{i} \}_{i \in \mathbb{Z}_{\ge 1}},$ such that 
$E_{j} \subset E_{i}~(i \le j)$ and each inclusions are bounded linear operators. 

(2)~A topological space $X$ is called {\it $C^{k}$-ILH-manifold modeled on $E$} 
if $X$ has the following (a) and (b):

(a)~$X$ is an inverse limit of $C^{k}$-Hilbert manifolds $\{ X_{i} \}_{i \in \mathbb{Z}_{\ge 1}}$ modeled on
$E_{i}$ such that $X_{j} \subset X_{i}~(i \le j),$

(b)~For each $x \in X$ and $i,$ there is an open neighborhood $X_{i} \supset U_{i}(x)$ and homeomorphism $\psi_{i}$
from $U_{i}(x)$ onto an open subset $V_{i}$ in $E_{i}$ which gives a $C^{k}$-coordinate around $x$ in $X_{i}$ and satisfies
$U_{j}(x) \subset U_{i}(x)~(i \le j),~\psi_{i+1}(y) = \psi_{i}(y)$ for all $y \in U_{i+1}(x).$

(3)~Let $X$~be a~$C^{k}$-ILH-manifold( $k \ge 1$ ) and $TX_{i}$ the tangent budle of $X_{i}.$ The inverse limit of $\{ TX_{i} \}$ is called 
{\it ILH-tangent bundle of $X$}.

(4)~Let $X,Y$~be~$C^{k}$-ILH-manifolds.
A mapping ~$\phi:~X \rightarrow Y$ is called {\it $C^{l}$-ILH-differentiable~( $l \le k$ )}
if $\phi$ is an inverse limit of $C^{l}$-differentiable maping $\{ \phi_{i} \}_{i \in \mathbb{Z}_{\ge 1}}$ 
(that is, for each $i,$ there exists 
$j(i)$ and $C^{l}$-map $\phi_{i}:X_{j(i)} \rightarrow Y_{i}$ such that 
$\phi_{i}(x) = \phi_{i+1}(x)$ for all $x \in X_{j(i+1)}.)$

(5)~$X$ is a {\it ILH-manifold}
if $X$ is a $C^{k}$-ILH-manifold for all $k \ge 0.$

(6)~Let $X,Y$~be~$\mathrm{ILH}$-manifolds.
A mapping $\phi:X \rightarrow Y$ is called {\it ILH-differentiable}
if
$\phi:X \rightarrow Y$ is $C^{k}$-ILH-differentiable for all $k \ge 0.$

(7)~Let $T_{x}X_{i}$ be the tangent space of $X_{i}$ at $x$ and $T_{x}X$ the inverse limit of $\{ T_{x}X_{i} \}.$
Let
\[
D^{r}\phi_{i}(x):~\prod^{r}_{l = 1} T_{x}X_{j(i)} \longrightarrow T_{\phi(x)}Y_{i}
\]
be the $r$-th (Fr\'echet) derivative of $\phi_{i}$ at $x.$
Then,
$\{ D^{r}\phi_{i}(x) \}_{i \in \mathbb{Z}_{\ge 1}}$ has the inverse limit
\[
\lim_{\leftarrow} D^{r}\phi_{i}(x):~\prod^{r}_{l = 1} T_{x}X \longrightarrow T_{\phi(x)}Y.
\]
It is called {\it $r$-th derivative of $\phi$} and we denote it by $D^{r}\phi(x).$

\end{defi}

Let $M$ and $g_{0}$ be the same as in Section 2 and use the same notations there.
As in the closed case, 
$\mathscr{M} := \lim_{\leftarrow} \mathscr{M}^{r}$~,~$\mathrm{Diff}(M) := \lim_{\leftarrow} \mathrm{Diff}^{r}$
naturally become ILH-manifolds and the pullback-action $A:~\mathrm{Diff}(M) \times \mathscr{M} \longrightarrow \mathscr{M}$
is ILH-differentiable.
Moreover, for a fixed metric $g_{0}$ on $M$, 
since each $\mathscr{M}^{r}_{C}$ are a submanifold of $\mathscr{M}^{r-1}$,
$\mathscr{M}_{C} := \lim_{\leftarrow} \mathscr{M}^{r}_{C}$ is an ILH-submanifold of $\mathscr{M}$
and the inclusion $\mathscr{M}_{C} \hookrightarrow \mathscr{M}$ is $C^{\infty}$-differentiable.
And, $\mathrm{Diff}_{C_{0}} := \lim_{\leftarrow} \mathrm{Diff}^{r}_{C_{0}}$ is an ILH-submanifold of $\mathrm{Diff}(M)$ and
the inclusion $\mathrm{Diff}_{C_{0}} \hookrightarrow \mathrm{Diff}(M)$ is $C^{\infty}$-differentiable.
Similarly, $\mathscr{M}_{C^{(1)}_{0}} := \lim_{\leftarrow} \mathscr{M}^{r}_{C^{(1)}_{0}}$ is an ILH-submanifold of $\mathscr{M}$
and the inclusion $\mathscr{M}_{C^{(1)}_{0}} \hookrightarrow \mathscr{M}$ is $C^{\infty}$ -differenciable.
Note that the pull-back action $A:\mathrm{Diff}_{C_{0}} \times \mathscr{M}_{C^{(1)}_{0}} \rightarrow \mathscr{M}_{C^{(1)}_{0}}$ 
is also $C^{\infty}$-differentiable.
By the Sobolev embedding for fibre bundles over manifold with boundary, we obtain the following
(see \cite{aubin2013some}):

\begin{lemm}
\label{lemm3.1}

Let $E,~F$ be vector bundles over $M$ and let $f:~E \rightarrow F$ be a $C^{\infty}$-differentiable which preserves each fibers.

Let $s > \frac{n}{2}$ .Then the bundle map induced by $f$
\[
\phi:~H^{s}(E) \rightarrow H^{s}(F)~;~\phi(\alpha) := f \circ \alpha
\]
is $C^{\infty}$-differentiable.

\end{lemm}

\begin{proof}
Same as Lemma 1.1 in \cite{koiso1979decomposition}. See also {\cite[Theorem~11.3]{palais1968foundations}}.
\end{proof}

Using this lemma and that $\mathscr{M}^{r}_{C}$ is a submanifold of $\mathscr{M}$, 
the following hold in the same as in \cite{koiso1979decomposition}:

\begin{lemm}[{\cite[Proposition1.2~,~Corollary1.3~,~Corollary1.4]{koiso1979decomposition}}]
\label{lemm3.2}

Let $s > \frac{n}{2}$ .

(1)~ 
$D:~\mathscr{M}^{s+1}_{C} \times H^{s+1}(T^{p}_{q}) \rightarrow H^{s}(T^{p}_{q+1})~;~(g,\xi) \mapsto \nabla_{g} \xi$

is $C^{\infty}$ -differentiable, where $T^{p}_{q}$ is the type $(p,q)$ tensor bundle and $\nabla_{g}$ is the Levi-Civita connection with respect to $g$ .

(2)~
$\mathscr{M}^{s+1}_{C} \times H^{s+2}(M) \rightarrow H^{s}(M)~;~(g,f) \mapsto \Delta_{g} f := \nabla_{g} d f.$

is $C^{\infty}$ -differentiable.

(3)~Mappings listed below are $C^{\infty}$ -differentiable:
\[
\mathscr{M}^{s+2}_{C} \rightarrow H^{s}(S^{2}T^{*})~;~g \mapsto {\rm Ric_{g}}~(\mathrm{the~Ricci~curvature~of}~g).
\]

\[
\mathscr{M}^{s+2}_{C} \rightarrow H^{s}(M)~;~g \mapsto R_{g}~(\mathrm{the~scalar~curvature~of}~g).
\]

\[
\mathscr{M}^{s+2}_{C} \rightarrow H^{s + 1/2}(\partial M)~;~g \mapsto H_{g}~(\mathrm{the~mean~curvature~of}~g~\mathrm{along}~\partial M).
\]
\end{lemm}

\begin{rema}
Since $\mathscr{M}_{C_{0}}$ is an ILH submanifold of $\mathscr{M}$, the same statements hold in the above lemma replaced 
$\mathscr{M}_{C}$ by $\mathscr{M}_{C^{(1)}_{0}}$.
\end{rema}

Let $r > \frac{n}{2} + 4$.
We define

$\mathscr{M}^{r}_{C,1} := \bigl\{ g \in \mathscr{M}^{r}_{C} \bigm| {\rm Vol}_{g}(M) = 1 \bigr\}$,
~$\mathscr{M}_{C,1} := \bigcap_{r} \mathscr{M}^{r}_{C,1}$,

$\mathfrak{S}^{r}_{C^{(1)}_{0}} := \bigl\{ g \in \mathscr{M}^{r}_{C,1} \cap \mathscr{M}^{r}_{C^{(1)}_{0}} \bigm| R_{g} = {\rm const} \bigr\}$,
~$\mathfrak{S}_{C^{(1)}_{0}} := \bigcap_{r} \mathfrak{S}^{r}_{C^{(1)}_{0}}$,

$\check{\mathfrak{S}}_{C^{(1)}_{0}} := \Biggl\{ g \in \mathfrak{S}_{C^{(1)}_{0}} \Biggm| \frac{R_{g}}{n-1} \notin \mathrm{Spec} (-\Delta_{g} ; \mathrm{Neumann})
\Biggr\}$.

\noindent
For $\bar{g},g \in \mathscr{M}^{r}_{C,1},$ we define

\[
\Phi^{r}_{g~(\bar{g})}:~H^{r}_{\bar{g}}(M) \longrightarrow \langle (1,-1) \rangle^{\perp_{\bar{g}}} \oplus H^{r-\frac{3}{2}}(\partial M)
\]
by
\footnotesize
\begin{equation*}
\begin{split}
f &\mapsto \Phi^{r}_{g~(\bar{g})}(f) := \Biggl( (n-1)(\Delta_{g})^{2} f - R_{g} -\Delta_{g} f - \int_{M} \bigl\{ (n-1)(\Delta_{g})^{2} f + R_{g} \Delta_{g} f \bigr\}dv_{\bar{g}} \\
&\quad+ \int_{\partial M} \nu_{g} \{ -(n-1)\Delta_{g} f - R_{g} f \}ds_{\bar{g} |_{\partial M}}
,~\nu_{g} \{ -(n-1)\Delta_{g} f - R_{g} f \} \bigm|_{\partial M},
~ \nu_{g} (f) \bigm|_{\partial M} \Biggr),
\end{split}
\end{equation*}
\normalsize
where $H^{r}(TM),H^{r}(T \partial M)$ are defined by fixed $\bar{g},$
$\nu_{g}$ is the outer unit normal vector along $\partial M$ with respect to $g,$ $H^{r}_{\bar{g}}(M) := \bigl\{ f \in H^{r}(M) \bigm| \int_{M} f~dv_{\bar{g}} = 0 \bigr\},$
and

$\langle (1,-1) \rangle^{\perp_{\bar{g}}} := \bigl\{ (u,v) \in H^{r-4}(M) \oplus H^{r-\frac{5}{2}}(\partial M) \bigm| \bigl( (u,v),(1,-1) \bigr)_{L^{2}(\bar{g})} = 0 \bigr\}.$
Here, $ds_{\bar{g} |_{\partial M}}$ also denotes the volume measure with respect to $\bar{g}|_{\partial M}.$

From
Lemma \ref{lemm3.2} and the trace theorem({\cite[Appendix B]{hormander1985analysis}}), $(g,f) \mapsto \Phi^{r}_{g}(f)$
is a $C^{\infty}$ -differentiable map from $\mathscr{M}^{r}_{C,1} \times H^{r}_{\bar{g}}(M)$
to
$H^{r-4}_{g_{0}}(M) \oplus H^{r-\frac{5}{2}}(\partial M) \oplus H^{r-\frac{3}{2}}(\partial M).$

And we define
\[
\mathcal{K}^{r}_{C} := \bigl\{ g \in \mathscr{M}^{r}_{C,1} \bigm| \exists \bar{g} \in \mathscr{M}_{C,1}~~{\rm s.t.}~~\Phi^{r}_{g~(\bar{g})}~\mathrm{is~an~isomorphism} \bigr\}.
\] 
Then, the following holds:

\begin{lemm}
\label{lemm3.3}
$\mathcal{K}^{r}_{C}$ is an open subset of $\mathscr{M}^{r}_{C,1}.$
\end{lemm}

\begin{proof}

\[
g \mapsto \Phi^{r}_{g}
\]
is a diffeomorphism mapping $\mathscr{M}^{r}_{C,1}$ to $\mathcal{L} \bigl( H^{r}_{\bar{g}}(M),\langle (1,-1) \rangle^{\perp_{\bar{g}}} \oplus H^{r-\frac{3}{2}}(\partial M) \bigr)$ .

On the other hand,
the set of all isomorphisms is open in
\[
\mathcal{L} \bigl( H^{r}_{\bar{g}}(M),\langle (1,-1) \rangle^{\perp_{\bar{g}}} \oplus H^{r-\frac{3}{2}}(\partial M) \bigr)
\]
with respect to the operator-norm.
Hence
$\mathcal{K}^{r}_{C}$ is an open subset of $\mathscr{M}^{r}_{C,1}$.
\end{proof}

\begin{lemm}
\label{lemm3.4}
$\check{\mathfrak{S}}_{C_{0}} = \mathscr{M}_{C_{0}} \cap \mathcal{K}^{r}_{C} \cap \mathfrak{S}^{r}_{C_{0}}$
and
$\check{\mathfrak{S}}_{C^{1}_{0}} = \mathscr{M}_{C^{1}_{0}} \cap \mathcal{K}^{r}_{C} \cap \mathfrak{S}^{r}_{C^{1}_{0}}.$
\end{lemm}

\begin{proof}
Since the proof for $\check{\mathfrak{S}}_{C^{1}_{0}}$ is exactly the same as for $\check{\mathfrak{S}}_{C_{0}}$,  
we will only prove for $\check{\mathfrak{S}}_{C_{0}}.$

( $\subset$ )

Fix $g \in \check{\mathfrak{S}}_{C_{0}}$ .

\underline{surjectivity} ;

Firstly, we consider the case that $R_{g} \neq 0$.

Given $(F,G,H) \in \langle (1,-1) \rangle^{g} \oplus H^{r-\frac{3}{2}}(\partial M)$ ,
we consider two boundary value problem:
\begin{equation}
-\Delta_{g} u = F~,~\nu_{g}(u) \Bigm|_{\partial M} = G,  \label{eq:Neu1}
\end{equation}

\begin{equation}
-\Delta_{g} v -\frac{R_{g}}{n-1} v = u~,~\nu_{g}(v) \bigm|_{\partial M} = H, \label{eq:Neu2}
\end{equation}

where $u \in H^{r-2}(M),~v \in H^{r}_{g}(M).$

We firstly consider (\ref{eq:Neu2}). 
For a fixed positive constant $\alpha \in \mathbb{R}_{> 0},$ we consider 
\[
-\Delta_{g} v + \alpha v = u~,~\nu_{g}(v) \bigm|_{\partial M} = H.
\]
Thus we can show that there is a unique solution by using standard variational argument.

Let 
\[
\Tilde{L}_{g}:~H^{r}(M) \longrightarrow H^{r-2}(M) \oplus H^{r-\frac{3}{2}}(\partial M)
\]
\[
;~u \mapsto \bigl( - \Delta_{g} u + \alpha u~,~\nu_{g}(u) \bigm|_{\partial M} \bigr)
\]
be the operatopr corresponding to the above equation, then this is an elliptic operator in the sence of Definition 20.1.1 in \cite{hormander1985analysis}.

Therefore,
from Theorem 20.1.2 in \cite{hormander1985analysis}, $\Tilde{L}_{g}$ is a Fredholm operator.
Hence
$\mathrm{dim}~\mathrm{Ker}\Tilde{L}_{g} < \infty,$ $\mathrm{dim}~\mathrm{Coker}\Tilde{L}_{g} < \infty$ and $\mathrm{Im}\Tilde{L}_{g}$ is closed.
Moreover, because of the existence and uniqueness of the above boundary value problem,
$\mathrm{ind}(\Tilde{L}_{g}) = \mathrm{dim}~\mathrm{Ker}\Tilde{L}_{g}-\mathrm{dim}~\mathrm{Coker}\Tilde{L}_{g} = 0$ .

We shall back to (\ref{eq:Neu2}). We consider the corresponding operator:
\[
L_{g}:~H^{r}(M) \longrightarrow H^{r-2}(M) \oplus H^{r-\frac{3}{2}}(\partial M)
\]
\[
;~u \mapsto \bigl( -\Delta_{g} u - \frac{R_{g}}{n-1} u~,~\nu_{g}(u) \bigm|_{\partial M} \bigr),
\]
then, this operator is also an elliptic operator.
From Theorem 20.1.8 in \cite{hormander1985analysis}, $\mathrm{ind}(L_{g}) = \mathrm{ind}(\Tilde{L}_{g})$ and since $g \in \check{\mathfrak{S}}_{C_{0}},$
$\mathrm{Ker}L_{g} = \{ 0 \}.$ Hence $\mathrm{dim}~\mathrm{Coker}L_{g} = 0.$
Therefore $L_{g}$ is surjective.

Next, we consider (\ref{eq:Neu1}).
The ellipticity only depends on its principal symbol, thus (\ref{eq:Neu1}) is also elliptic
(Exactly speaking, the operator corresponding to (\ref{eq:Neu1}) is elliptic).
Let 
\[
\Check{L}_{g}:~H^{r}(M) \longrightarrow H^{r-2}(M) \oplus H^{r-\frac{3}{2}}(\partial M)
\]
\[
;~u \mapsto \bigl( -\Delta_{g} u~,~\nu_{g}(u) \bigm|_{\partial M} \bigr)
\]
be the corresponding operator, then from Theorem 20.1.8 in \cite{hormander1985analysis} and the above things,
\[
0 = \mathrm{ind}(\Tilde{L}_{g}) = \mathrm{ind}(\Check{L}_{g}).
\]
Since $\mathrm{Ker}\Check{L}_{g} = \mathbb{R}~( = \{ \mathrm{constant~functions} \}),$ $\mathrm{dim}~\mathrm{Ker}\Check{L}_{g} = 1.$

\noindent
Thus $\mathrm{dim}\mathrm{Coker}\Check{L}_{g} = 1.$
On the other hand, from the Green's formula,
\[
(F,G) \in \mathrm{Im}~\Check{L}_{g} \Rightarrow \int_{M} F~dv_{g}-\int_{\partial M} G~ds_{g|_{\partial M}} = 0.
\]
Hence $\mathrm{Coker}\Check{L}_{g} \cong \langle (1,-1) \rangle.$
Thus $\mathrm{Im}\Check{L}_{g} = \langle (1,-1) \rangle^{\perp_{g}}.$
Therefore $\Phi^{r}_{g~(g)}$ is surjective if $R_{g} \neq 0.$

Next, we consider the case that $R_{g} = 0.$
The above observation implies that (\ref{eq:Neu1}) has a unique solution up to constants.
That is, if we take a solution $u(F,G)$ of (\ref{eq:Neu1}), then $u(F,G) + C$ ( $C$ is arbitrary constant) is also solution of this equation.
Hence, for given $H$, we take a constant $C$ so that
\[
\int_{M} u(F,G) + C~dv_{g} = - \int_{\partial M} H~ds_{g|_{\partial M}}.
\]
Then, from the above observation, there exists a solution $v$ of (\ref{eq:Neu2}).
Therefore $\Phi^{r}_{g~(g)}$ is also surjective if $R_{g} = 0$. 

\underline{injectivity} ;

Let $\Phi^{r}_{g}(u) = 0$ , then
\begin{equation}
(n-1)(\Delta_{g})^{2} u + R_{g} \Delta_{g}u = 0,  \label{eq:inj1}
\end{equation}
\begin{equation}
\nu_{g} \bigl\{ -(n-1)\Delta_{g} u - R_{g} u \bigr\} \bigm|_{\partial M} = 0,  \label{eq:inj2}
\end{equation}
\begin{equation}
\nu_{g}(u) \bigm|_{\partial M} = 0.  \label{eq:inj3}
\end{equation}

Multiply the contents in \{ \} of (\ref{eq:inj2}) by the left hand side of (\ref{eq:inj1}) and integration it over $M.$
Thus, by integration by parts, (\ref{eq:inj1}) and (\ref{eq:inj2}),
\[
-(n-1)\Delta_{g}u - R_{g}u = \mathrm{const} .
\]
On the other hand,
\[
\bigintss_{M} -(n-1)\Delta_{g}u - R_{g}u~dv_{g} = \bigintss_{\partial M} (n-1) \nu_{g}(u)~ds_{g|_{\partial M}}-R_{g} \bigintss_{M} u~dv_{g}.
\]
Hence, from (\ref{eq:inj3}) and $u \in H^{r}_{g}(M)$, the first term of the right hand side is zero. Thus, since $u \in H^{r}_{g}$ , the second term also vanish.
Therefore, we have
\begin{equation}
-(n-1)\Delta_{g} u - R_{g} u = 0.  \label{eq:inj4}
\end{equation}
Hence, if $R_{g} \neq 0$, then $u = 0$.

On the other hand, if $R_{g} = 0$, we multiply the both sides of (\ref{eq:inj4}) by $u$, integral over $M,$
use the integration by parts and get $u = \mathrm{const}.$ But, since $u \in H^{r}_{g}(M)$, $u \equiv 0.$  
Therefore $\Phi^{r}_{g~(g)}$ is injective.

( $\supset$ ) This inclusion is obvious.
\end{proof}

\begin{lemm}
\label{lemm3.5}
$\check{\mathfrak{S}}_{C^{1}_{0}} \neq \emptyset.$
\end{lemm}

\begin{proof}

Since $\mathrm{dim} M \ge 3,$ we can construct a negative constant scalar curvature metric $g$ (see \cite{cruz2019prescribing}).
Thus, since the eigenvalues of $-\Delta_{g}$ are positive(\cite[Theorem~4.4]{aubin2013some}), from Lemma~\ref{lemm3.4}, 
$g \in \check{\mathfrak{S}}_{C^{1}_{0}}.$
\end{proof}

\begin{lemm}
\label{lemm3.6}

$\mathfrak{S}^{r}_{C^{(1)}_{0}} \cap \mathcal{K}^{r}_{C}$ is an ILH-submanifold of $\mathscr{M}^{r}_{C,1}.$
\end{lemm}

\begin{proof}
For each $g \in \mathfrak{S}^{r}_{C_{0}} \cap \mathcal{K}^{r}_{C}$ , we define a map
\[
\Psi:~\mathscr{M}^{r}_{C,1} \longrightarrow \langle (1,-1) \rangle^{\perp_{\bar{g}}} \oplus H^{r-\frac{3}{2}}(\partial M)
\]
by
\small
\[
g \mapsto \Bigl( -\Delta_{g} R_{g} + \int_{M} \Delta_{g} R_{g}~dv_{g_{0}} + \int_{\partial M} \nu_{g}(R_{g})~ds_{g_{0} \bigm|_{\partial M}}, \nu_{g} (R_{g}) |_{\partial M}, \frac{2}{n-1} H_{g}|_{\partial M} \Bigr).
\]
\normalsize
From Lemma \ref{lemm3.2}, this is a $C^{\infty}$-differentiable map.

We note that $\Psi^{-1}(0) = \mathfrak{S}^{r}_{C_{0}}$ . 
In fact, if $g \in \mathfrak{S}^{r}_{C_{0}}$, then $H_{g} = 0$ along $\partial M$ 
since $g \in \mathscr{M}^{r}_{C_{0}}$.
And it is clear that the first two terms are zero if $g \in \mathfrak{S}^{r}_{C_{0}}$, 
hence the inclusion of $\Psi^{-1}(0) \supset \mathfrak{S}^{r}_{C_{0}}$ holds.
 
On the other hand, for $g \in \Psi^{-1}(0)$, then, since the first and second components are both zero,
\[
\Delta_{g} R_{g} = \mathrm{const}~.
\]
But, since the second component is zero and from the Green's formula, this constant must be zero.
Thus, by multiplying $\Delta_{g} R_{g}$ by $R_{g}$ and integrating it over $M,$ from integration by parts and the fact that the second component is zero,
\[
0 = \int_{M} \bigl| \nabla_{g} R_{g} \bigr|^{2}_{g}~dv_{g}.
\]
Hence $R_{g} = \mathrm{const}.$
Since the third component is zero, we obtain $g \in \mathscr{M}^{r}_{C_{0}}.$

The derivative of $\Psi$ at $g \in \mathfrak{S}^{r}_{C_{0}} \cap \mathcal{K}^{r}_{C}$ is calculated as follows(see {\cite[Theorem~1.174]{besseeinstein}}):
\footnotesize
\begin{equation*}
\begin{aligned}
&D_{g}\Psi(h) \\
&= \Biggl( -\Delta_{g} \bigl( -\Delta_{g} tr_{g}h + \delta_{g}\delta_{g}h - \langle h, {\rm Ric_{g}} \rangle_{g} \bigr) 
- \int_{M} \Delta_{g} \bigl( \Delta_{g} tr_{g}h - \delta_{g}\delta_{g}h + \langle h, {\rm Ric_{g}} \rangle_{g} \bigr) dv_{g_{0}} \\ 
&- \int_{\partial M} \nu_{g} \bigl( \Delta_{g}tr_{g}h - \delta_{g}\delta_{g}h + \langle h, {\rm Ric_{g}} \rangle_{g} \bigr) ds_{g_{0} |_{\partial M}}, 
\nu_{g} \bigl( -\Delta_{g}tr_{g}h + \delta_{g}\delta_{g}h - \langle h, {\rm Ric_{g}} \rangle_{g} \bigr) |_{\partial M}, \\
&\frac{1}{n-1} \bigl( [ d(tr_{g} h) - \delta_{g} h ](\nu_{g}) - \delta_{g|_{\partial M}} (h(\bullet, \nu_{g})) - g_{\partial M}(\Pi_{g}, h) \bigr)|_{\partial M}  \Biggr),
\end{aligned}
\end{equation*}
\normalsize
where $\Pi_{g}$ denotes the second fundamental form of $(\partial M, g).$
Take the variation $h = fg~(f \in H^{r}_{g}(M)).$
Then, since $g \in \mathcal{K}^{r}_{C},$
we get $D_{g}\Psi$ is surjective.
Therefore, from the Inverse function theorem~(\cite{omori2017infinite}),
$\mathfrak{S}^{r}_{C_{0}} \cap \mathcal{K}^{r}_{C}$ is a submanifold of $\mathscr{M}^{r}_{C,1}$ 
and the tangent space at $g \in \mathfrak{S}^{r}_{C_{0}} \cap \mathcal{K}^{r}_{C}$
is $\mathrm{Ker} D_{g}\Psi$.

On the other hand, since $\mathfrak{S}^{r}_{C^{1}_{0}} \cap \mathcal{K}^{r}_{C}$
is an ILH submanifold of $\mathfrak{S}^{r}_{C_{0}} \cap \mathcal{K}^{r}_{C}$,
$\mathfrak{S}^{r}_{C^{1}_{0}} \cap \mathcal{K}^{r}_{C}$
is an ILH submanifold of $\mathscr{M}^{r}_{C, 1}.$ 
\end{proof}

\begin{lemm}
\label{lemm3.7}

Let $C^{r}_{+}(M)_{N} := \bigl\{ f \in H^{r}(M) \bigm| f > 0~\mathrm{on}~M,~\nu_{g_{0}}(f) = 0~\mathrm{on}~\partial M \bigr\}$ and
be a map 
\[
\chi^{r}:~C^{r}_{+}(M)_{N} \times \bigl( \mathfrak{S}^{r}_{C^{1}_{0}} \cap \mathcal{K}^{r}_{C} \bigr) \longrightarrow \mathscr{M}^{r}_{C^{1}_{0}}~;~(f,g) \mapsto f \cdot g.
\]
Then $\chi^{r}$ is $C^{\infty}$-differentiable. 
Moreover, if $g \in \check{\mathfrak{S}}_{C^{1}_{0}},$
then $D_{(f,g)} \chi^{r}$ is an isomorphism.
\end{lemm}

\begin{proof}
It is clear that this is a $C^{\infty}$-differentiable map. In fact,
\[
D_{(f,g)}\chi^{r}(\phi,h) = f h + \phi g.
\]

\underline{injectivity} ;

Let $f h + \phi g = 0$ ,then since
$\mathrm{Ker}D_{g}\Psi \in h = - \frac{\phi}{f} g =: \tilde{\phi} g,$
$\tilde{\phi} \in \mathrm{Ker}\Phi^{r}_{g}.$
On the other hand, since
$g \in \mathcal{K}^{r}_{C},$
$\tilde{\phi} = 0.$
Hence, since $f \neq 0,$ $\phi = 0~,~h = 0.$

\underline{surjectivity} ;

We shall show it by contradiction. 

If $D_{(f,g)}\chi^{r}$ is not surjective, then $\exists \bar{h} \in \bigl( \mathrm{Im}D_{(f,g)} \chi^{r} \bigr)^{\perp_{g}} \setminus \{ 0 \}$

(since 
\[
\mathrm{Im}D_{(f,g)} \chi^{r} = f T_{g} \bigl( \mathfrak{S}^{r}_{C^{1}_{0}} \cap \mathcal{K}^{r}_{C} \bigr) + H^{r}(M) g
\]
is a closed subspace in $T_{fg}\mathscr{M}^{r}_{C^{1}_{0}}$ ).
 
We define an operator on $\bigl( H^{r}(M) g \bigr)^{\perp_{g}}$ (which is a closed subspace in $T_{fg}\mathscr{M}^{r}_{C^{1}_{0}}$)
\[
K_{g}(h) := -\Delta_{g} tr_{g}h + \delta_{g}\delta_{g}h - \langle h, {\rm Ric_{g}} \rangle_{g}.
\]

From the Green's formula, 
\begin{equation*}
\begin{split}
\bigl( K_{g}h,f \bigr)_{M} - \bigl( h,K^{*}_{g}f \bigr)_{M}
=\bigl( &\nu_{g}(f) , tr_{g}h \bigr)_{\partial M} - \bigl( f , \nu_{g} (tr_{g}h) \bigr)_{\partial M} \\
&\quad+ \bigl( -h(\nu_{g} , \bullet) , \nabla_{g}f \bigr)_{\partial M} + \bigl( -\delta_{g}h(\nu_{g}) , f \bigr)_{\partial M},
\end{split}
\end{equation*}
where $K^{*}_{g}$ is the formal adjoint of $K_{g}:$
\[
K^{*}_{g} f := -\bigl( \Delta_{g} f \bigr)g + \nabla_{g}\nabla_{g}f - f  {\rm Ric_{g}}.
\]
Since $K_{g}$ is defined on $\bigl( H^{r}(M) g \bigr)^{\perp_{g}}$ and 
$h \in T_{fg} \mathscr{M}^{r}_{C^{1}_{0}},~tr_{g} h \equiv 0$ on $\partial M.$
Therefore $h = \frac{tr_{g} h}{n} g_{0} = 0$ on $\partial M.$
Hence the first three terms in the right hand side of the above equation vanish on $\partial M$.
Also, since $h \in T_{fg} \mathscr{M}^{r}_{C^{1}_{0}}, D_{fg}H(h) = 0,$
where $D_{fg}H$ denotes the derivative of the mean curvature $H$ at $fg.$
And 
\[
D_{fg} H (h) = \frac{1}{2} \bigl( [ d(tr_{fg} h) - \delta_{fg} h ] (\nu_{fg}) -\delta_{(fg)|_{\partial M}} (h(\bullet, \nu_{fg})) 
- (fg)|_{\partial M} (\Pi_{fg}, h ) \bigr),
\]
(cf. {\cite[Claim~3.1]{akutagawarelative}},~{\cite[Section~2]{cruz2019prescribing}})
where $\Pi_{fg}$ denotes the second fundamental form on $(\partial M, fg)$.
Since $h \in T_{fg} \mathscr{M}^{r}_{C^{1}_{0}}$ and $tr_{g} h = 0$, we obtain $\delta_{fg} h (\nu_{fg}) = 0.$
Take a point $p \in \partial M.$ Since $h|_{\partial M} = ( n^{-1}tr_{fg} h )g|_{\partial M} = 0$, 
$\delta_{g} h (\nu_{g}) = d h_{i 0} (x^{i})$ at $p,$
where $x^{i}$ and $h_{i 0}$ denote respectively a local normal coordinates at $p$ 
with respect to $g$ such that $\nu_{g}(p) = \frac{\partial}{\partial x^{0}}(p)$ and 
the $(i, 0)$-th components of $h$ with respect to $(x^{i}).$
Let $(y^{j})$ be a local normal coordinates at $p$ with respect to $fg$ such that $\nu_{fg}(p) = \frac{\partial}{\partial y^{0}}(p)$ and $\tilde{h}_{ij}$ the components of $h$ with respect to $(y^{j}).$
Then, 
$d h_{i 0} (x^{i})(p)
= f^{1/2} \frac{\partial (f \tilde{h}_{i 0})}{\partial y^{i}} (p)
= \bigl( f^{3/2} \frac{\partial \tilde{h}_{i 0}}{\partial y^{i}} 
+ \nu_{g} f \tilde{h}_{0 0}
+ \sum^{n}_{\alpha = 1} \frac{\partial f}{\partial x^{\alpha}} \tilde{h}_{\alpha 0} \bigr) (p)
= f^{3/2} \delta_{fg} h (\nu_{fg}) (p)
= 0.$
Hence $\delta_{g} h(\nu_{g}) = 0~\mathrm{on}~\partial M$ and 
we get
\[
\bigl( K_{g}h,f \bigr)_{M} = \bigl( h,K^{*}_{g}f \bigr)~,~\forall h \in \bigl( H^{r}(M)_{N} g \bigr)^{\perp_{g}},\forall f \in H^{r-2}(M).
\]
Thus, from the proof of Lemma \ref{lemm2.16} (and the fact that the principal symbol of $K_{g}$ is surjective),
we can get an orthogonal decomposition(cf. the proof of Lemma 2.4 in \cite{koiso1979decomposition}):
\[
T_{fg}\mathscr{M}^{r}_{C} = H^{r}(M)g \oplus \mathrm{Ker}K_{g} \oplus \mathrm{Im}K^{*}_{g}.
\]
From the hypothesis $f\bar{h} \in \Bigl( T_{g} \bigl( \mathfrak{S}^{r}_{C_{0}} \cap K^{r}_{g} \bigr) \Bigr)^{\perp_{g}},$
$f\bar{h} \in \bigl( H^{r}(M)g \bigr)^{\perp_{g}},$ and from the above decomposition, $f\bar{h} \in \mathrm{Im}K^{*}_{g}$
(since, if $fg \in \mathrm{Ker} K_{g}$, then it must be $fg \in \mathrm{Ker} \Psi^{r}_{g} = T_{g}(\mathfrak{S}^{r}_{C_{0}} \cap \mathcal{K}^{r}_{C})$).
Let $f\bar{h} = K^{*}_{g}(\psi)$ , then
\[
f\bar{h} = -\bigl( \Delta_{g} \psi \bigr)g + \nabla_{g}\nabla_{g} \psi - \psi {\rm Ric_{g}}.
\]
But, since $f\bar{h} \in \bigl( H^{r}(M)g \bigr)^{\perp_{g}}$,
\[
0 = tr_{g}(f\bar{h}) = -(n-1) \Delta_{g} \psi - R_{g} \psi.
\]
Thus we can see that the image of $K_{g}$ is included in $H^{r-2}(M)_{N}$.
Then, from the above equation and $g \in \check{\mathfrak{S}}_{C^{1}_{0}}$, we obtain $\psi \equiv 0$.
Hence $f\bar{h} \equiv 0$.
This contradicts that $f\bar{h} \neq 0$ .
Therefore $D_{(f,g)} \chi^{r}$ is surjective.

We will show that the image of $K_{g}$ is included in $H^{r-2}(M)_{N}$ in the following.
As in the proof of lemma \ref{lemm2.16}, there is a decomposition of $T_{fg} \mathscr{M}^{r}_{C^{1}_{0}}$
\[
T_{fg}\mathscr{M}^{r}_{C_{0}} = \mathrm{Im} \bigl( \alpha |_{T_{id.} \mathrm{Diff}^{r+1}_{C_{0}}} \bigr) \oplus \mathrm{Ker} \bigl( \alpha^{*} |_{T_{g}\mathscr{M}^{r}_{C_{0}}} \bigr),
\]
where $\alpha$ is the Lie derivative of $g$ and $\alpha^{*}$ is the divergence operator with respect to $g$.
Therefore, we can write that $h = h_{1} + h_{2},~h_{1} \in \mathrm{Im} \bigl( \alpha |_{T_{id.} \mathrm{Diff}^{r+1}_{C_{0}}} \bigr),~
h_{2} \in \mathrm{Ker} \bigl( \alpha^{*} |_{T_{g}\mathscr{M}^{r}_{C^{1}_{0}}} \bigr)$.

We firstly consider $h_{1}$.
Let $h_{1} = \mathcal{L}_{X} g,~X \in H^{r+1}(TM)$.
Since $K_{g}$ is the first derivative of the functional $g \mapsto R_{g}$, 
$R_{g}$ is diffeomorphism invariant and 
the derivative of the pull-back action of diffeomorphism on $g$ is $\mathcal{L}_{\bullet} g$ (as mentioning in the proof of lemma \ref{lemm2.9}),
\[
K_{g} h_{1} = 0~\mathrm{on}~M.
\] 
Hence $\nu_{g_{0}} (K_{g}h_{1}) = 0~\mathrm{on}~\partial M$.

Finally, we consider $h_{2} \in \mathrm{Ker}~\delta_{g}$.
Then $\delta_{g} \delta_{g} h_{2} = 0~\mathrm{on}~M$.
Since $h_{2} \in T_{fg} \mathscr{M}^{r}_{C_{0}}$,
we can write $j^{1}_{\partial M}h_{2} = j^{1}_{\partial M} (\rho \cdot g_{0})$
for some
$\rho \in H^{r - 1/2}(C^{\infty}(M))$.
Note that $\nu_{g_{0}} \rho = 0$ on $\partial M.$
In fact, since $\nabla_{\nu_{g_{0}}} g_{0} \equiv 0,$
$(\nabla_{\nu_{g_{0}}} \rho) g_{0} = \nabla_{\nu_{g_{0}}} h_{2}$ on $\partial M$.
Hence, $\nabla_{\nu_{g_{0}}} \rho = (1/n) g_{0}^{ij} \nabla_{\nu_{g_{0}}} (h_{2})_{ij}$ on $\partial M,$
where $g_{0}^{ij}$ and $(h_{2})_{ij}$ denote respectively the components of $g^{-1}$ and $h$ in terms of some local coordinates. 
On the other hand, $g_{0}^{ij} \nabla_{\nu_{g_{0}}} (h_{2})_{ij} = - \nabla_{\nu_{g_{0}}} (g_{0}^{ij}) (h_{2})_{ij} = 0$ on $\partial M$
since $tr_{g_{0}} h_{2} = 0$ on $M$ and $h_{2}|_{\partial M}~(= \frac{1}{n} \rho|_{\partial M})= 0$.
Therefore, on $\partial M$,
\[
\nabla_{\nu_{g_{0}}} \bigl( \langle h_{2}, {\rm Ric_{g}} \rangle_{g} \bigr) = \langle \nabla_{\nu_{g_{0}}} (\rho g_{0}), {\rm Ric_{g}} \rangle_{g}
+ \langle \rho g_{0}, \nabla_{\nu_{g_{0}}} {\rm Ric_{g}} \rangle_{g}~\mathrm{on}~\partial M,
\]
where $\langle , \rangle_{g}$ denote the natural inner product on the (0,2)-tensor bundle of $M$ induced by $g.$
Here, as mentioning above, $\rho = 0~\mathrm{on}~\partial M$.
And $\nabla_{\nu_{g_{0}}}(\rho g_{0}) = (\nu_{g_{0}}(\rho)) g_{0} + \rho \nabla_{\nu_{g_{0}}} g_{0} = 0~\mathrm{on}~\partial M$
since $\nu_{g_{0}} (\rho) = 0~\mathrm{on}~\partial M$ and $\nabla_{\nu_{g_{0}}} g_{0} \equiv 0~\mathrm{on}~M$.
Consequently, we obtain $\nu_{g_{0}} (K_{g} h_{2}) = 0~\mathrm{on}~\partial M.$
\end{proof}

We get the following lemma in the same way as Lemma~2.8 in \cite{koiso1978nondeformability}:

\begin{lemm}
\label{lemm3.8}

Let $E$ and $F$ be bector bundles over $M$ associated with the frame bundle. Any $\eta \in \mathrm{Diff}_{C_{0}}$
defines a natural linear map (by pullback)
\[
\eta^{*}:~H^{k}(E) \longrightarrow H^{k}(E)~~(k \ge n/2 + 2).
\]
Let $r \ge n/2 + 2,~A \subset H^{r}(E)$ be an open subset and let
$\phi:~A \rightarrow H^{r}(F)$ be a $C^{\infty}$ -differentiable map which commutes with the action of $\mathrm{Diff}_{C_{0}}$ .
Put $A^{s} := A \cap H^{s}(E)~(s \ge r)$ .
Then $\phi \bigl( A^{s} \bigr) \subset H^{s}(F)$ and $\phi |_{A^{s}}:~A^{s} \rightarrow H^{s}(F)$ is $C^{\infty}$ -differentiable.

\end{lemm}

\begin{theo}
\label{theo3.1}

$\check{\mathfrak{S}}_{C^{1}_{0}}$ is an ILH-submanifold of $\mathscr{M}_{C^{1}_{0}}$ and the map
\[
\chi:~C^{\infty}_{+}(M)_{N} \times \check{\mathfrak{S}}_{C^{1}_{0}} \longrightarrow \mathscr{M}_{C^{1}_{0}}~;~(f,g) \mapsto f \cdot g
\]
is a local ILH-diffeomorphism into $\mathscr{M}_{C^{1}_{0}}$.
\end{theo}

\begin{proof}
It can be proved exactly the same as the proof of Theorem 2.5 in \cite{koiso1979decomposition} using lemma~\ref{lemm3.7} and \ref{lemm3.8}.
And note that $\bigcap_{r} C^{r}_{+}(M)_{N} = C^{\infty}_{+}(M)_{N}$.
\end{proof}

Since $\mathrm{Diff}_{C_{0}}$ and $\mathscr{M}_{C^{(1)}_{0}}$ are submanifold of $\mathrm{Diff}(M)$ and $\mathscr{M}$ respectively,
from Lemma~\ref{lemm3.8} and Theorem~\ref{theo2.1}, we can obtain the following $C^{\infty}$-version of the Slice theorem
exactly same as in \cite{koiso1978nondeformability}:

\begin{theo}[$C^{\infty}$-version of Theorem~\ref{theo2.1}]
\label{theo3.2}

For all $g \in \mathscr{M}_{C^{(1)}_{0}}$ there exists an ILH-submanifold $\mathcal{S}_{g} \subset \mathscr{M}_{C^{(1)}_{0}}$ containing $\gamma$ 
so that the following holds:

(1) $\eta \in I_{g} \Rightarrow \eta^{*} \mathcal{S}_{g} = \mathcal{S}_{g},$

(2) $\eta \in \mathrm{Diff}_{C_{0}}~,~\eta^{*} \bigl( \mathcal{S}_{g} \bigr) \cap \mathcal{S}_{g} \neq \emptyset \Rightarrow \eta \in I_{g}$ and 

(3) There exists a local section defined on an open neighborhood of $[I_{g}]:$
\[
\exists \chi~:~\bigl( \mathrm{Diff}_{C_{0}} / I_{g} \supset \bigr)U \longrightarrow \mathrm{Diff}_{C_{0}}
\]
such that
\[
F~:~U \times \mathcal{S}_{g} \longrightarrow \mathscr{M}_{C^{(1)}_{0}}~;~(u,t) \mapsto \chi(u)^{*} t 
\]
is an ILH-diffeomorphism mapping onto an open nighborhood of $g.$

\end{theo}

Consequently, from this Slice theorem and Theorem~\ref{theo3.1}, 
we can prove Main Theorem in Section 1.

\begin{proof}[Proof of Main Theorem]
From Theorem \ref{theo3.1}, we can decompose as
\[
g(t) = f(t) \tilde{g}(t),
\]
where $f(t)$ is a deformation of $f$ in $C^{\infty}_{+}(M)_{N}$ and
$\tilde{g}(t)$ is a deformation of $\bar{g}$ in $\check{\mathfrak{S}}_{C^{1}_{0}}$.
Moreover, from Theorem~\ref{theo3.2}, $\tilde{g}(t)$ can be decomposed as
\[
\tilde{g}(t) = \phi(t)^{*} \bar{g}(t)~\mathrm{with}~\delta \bar{g}^{'}(0) = 0.
\]

Since the scalar curvature is invariant under the action of diffeomorphisms,
\[
R_{\tilde{g}(t)} = R_{\bar{g}(t)} \equiv \mathrm{const}.
\]
\end{proof}

\begin{theo}
\label{coro3.2}

For any $g = f\bar{g}~( f \in C^{\infty}_{+}(M)_{N},~\bar{g} \in \check{\mathfrak{S}}_{C^{1}_{0}} )$ and any smooth deformation
$\{ g(t) \}_{t \in (-\epsilon,\epsilon)} ( \subset \mathscr{M}_{C^{1}_{0}} )$ of $g$ for sufficiently small $\epsilon > 0,$
there exists uniquely a smooth deformation $\{ f(t) \}_{t \in (-\epsilon,\epsilon)} ( \subset C^{\infty}_{+}(M)_{N} )$ of $f,$
a smooth one $\{ \phi(t) \}_{t \in (-\epsilon,\epsilon)} ( \subset \mathrm{Diff}_{C_{0}} )$ of $id_{M}$
and a smooth one $\{ \bar{g}(t) \}_{t \in (-\epsilon,\epsilon)} ( \subset \check{\mathfrak{S}}_{C^{1}_{0}} )$ of $\bar{g}$
with $\delta \bigl( f^{'}(0)\bar{g} + f \bar{g}^{'}(0) \bigr) = 0$
such that
\[
g(t) = \bigl( f(t) \circ \phi(t) \bigr) \phi(t)^{*} \bar{g}(t).
\]

\end{theo}

\begin{proof}
Reverse the order of applying Theorem~\ref{theo3.1} and Theorem~\ref{theo3.2} in the proof of Main Theorem.
\end{proof}

\section{Applications}
~We use the same notations as those in the above sections.
We give two applications of the following.

\subsection{Some rigidity theorems for relative constant scalar curvature metrics}
~~In the case of $\partial M = \emptyset,$ 
a metric $g$ in a given conformal class $C$ is called a {\it Yamabe~metric}
if $g$ is a minimizer of the restriction $\mathcal{E}|_{C}$
of the normalized Einstein-Hilbert functional $\mathcal{E}.$
The infimum of $\mathcal{E}|_{C}$ is called the {\it Yamabe constant} $Y(M, C)$ of $C.$
By combining the Koiso's decomposition theorem 
with the existence of a Yamabe metric in each conformal class,
B\"{o}hm-Wang-Ziller proved the following (see the proof of {\cite[Theorem~5.1]{bohm2004variational}}):
\begin{theo}
\label{theo4.0}
Let $(M^{n}, g_{\infty})$ be a closed Riemannian manifold of
dimension $n \ge 3.$
Assume that $g_{\infty}$ is a unique constant scalar curvature ({\it csc} metric for brevity)
in its conformal class  up to rescaling
with $\lambda_{1}(-\Delta_{g_{\infty}}) > \frac{R_{g_{\infty}}}{n-1}.$
Here, $\lambda_{1}(-\Delta_{g_{\infty}})$ denotes the first non-zero eigenvalue of $-\Delta_{g_{\infty}}.$
Then each csc metric sufficiently close to $g_{\infty}$
with respect to the $C^{\infty}$-topology is a Yamabe metric in its conformal class. 
\end{theo}

\begin{rema}
In the above, the condition $\lambda_{1}(-\Delta_{g_{\infty}}) > \frac{R_{g_{\infty}}}{n-1}$
implies that $(M^{n}, [g_{\infty}])$ is not conformally equivalent to the standard $n$-sphere $(S^{n}, [g_{{\rm std}}]).$
\end{rema}

On the other hand, using a compactness theorem of the space of all csc metrics in a fixed conformal class
(proved by Khuri-Maqrues-Schoen~{\cite[Theorem~1.1]{khuri2009compactness}}),
one can also get the following:
\begin{theo}
Let $(M^{n}, g_{\infty})$ be a closed Riemannian manifold of
dimension either $3 \le n \le 7,$ or both $8 \le n \le 24$ and that $M$ is spin.
Assume that $g_{\infty}$ is a unique csc metric
in its conformal class  up to rescaling 
with $\lambda_{1}(-\Delta_{g_{\infty}}) > \frac{R_{g_{\infty}}}{n-1}$.
Then each csc metric sufficiently close to $g_{\infty}$
with respect to the $C^{\infty}$-topology is also a {\it unique} csc metric
up to rescaling in its conformal class.
\end{theo}

 We can prove a similar statement below on a manifold with boundary.
When $\partial M \neq \emptyset,$ for $g \in \mathscr{M}$ with $H_{g} = 0$ along $\partial M,$
\[
Y(M , [g]_{0}) := \inf_{h \in [g]_{0}} \mathcal{E}(h)
\]
is called the {\it relative Yamabe constant} of $[g]_{0}.$
A metric $h$ with $H_{h} = 0$ along $\partial M$ is called
a {\it relative Yamabe metric}
if $Y(M, [g]_{0}) = \mathcal{E}(h)$
(see \cite{akutagawarelative} for more details).
Here, $[g]_{0}$ denotes the {\it relative conformal class} of $g$, that is,
$[g]_{0} := \bigl\{ h \in [g] \bigm| H_{h} = 0~\mathrm{along}~\partial M \bigr\} 
= \bigl\{ u \cdot g \bigm| u \in C^{\infty}_{+}(M),~\nu_{g}(u)|_{\partial M}  = 0 \bigr\}.$
Then we can prove the following:
\begin{theo}
\label{theo4.1}
Let $(M^{n}, g_{\infty})$ be a compact connected Riemannian manifold of dimension $n \ge 3$ with smooth non-empty minimal boundary $\partial M$ 
(i.e., $H_{g_{\infty}} = 0$ along $\partial M$).
Assume that $g_{\infty}$ is a unique relative csc metric in $[g_{\infty}]_{0}$ up to rescaling with
$\lambda_{1}(-\Delta_{g_{\infty}} ; \mathrm{Neumann}) > \frac{R_{g_{\infty}}}{n-1},$
where $\lambda_{1}(-\Delta_{g_{\infty}} ; \mathrm{Neumann})$ denotes the first non-zero eigenvalue of
$-\Delta_{g_{\infty}}$ with the Neumann boundary condition (see {\cite[Proposition~2.6]{akutagawa2001notes}} and {\cite[Section~1]{schoen1989variational}}) 
Moreover, we assume the following: either

(a) $(M^{n}, g_{\infty})$ has a nonumbilic point on $\partial M,$ 

\noindent
or

(b) $(M^{n}, g_{\infty})$ is umbilic boundary (therefore, $\partial M$ is totally geodesic)
satisfying that one of the followings (b1)-(b3) holds:

 (b1) the Weyl tensor does not vanish identically on $\partial M$ and $n \ge 6,$
 
 (b2) $M$ is locally conformally flat,
 
 (b3) $n = 3, 4,$ or $5.$

\noindent
Then each relative csc metric sufficiently close to $g_{\infty}$ in $\mathscr{M}_{C^{1}_{0}}$
with respect to he $C^{\infty}$-topology is a relative Yamabe metric.
\end{theo}

\begin{rema}
Escobar \cite{escobar1992yamabe} proved that, for any $(M, g_{\infty})$
satisfying either (a) or (b) in the above,
then there exists a relative Yamabe metric in $[g_{\infty}]_{0}.$
Note also that, the condition $\lambda_{1}(-\Delta_{g_{\infty}} ; \mathrm{Neumann}) > \frac{R_{g_{\infty}}}{n-1}$
implies that $(M, [g_{\infty}]_{0})$ is not conformally equivalent to the standard hemisphere $(S^{n}_{+}, [g_{{\rm std}}]_{0}).$
\end{rema}

The following follows directly from Main Theorem.

\begin{coro}
\label{theo4.2}
Assume that $g_{\infty} \in \mathfrak{S}_{C^{1}_{0}}$ satisfies 
$\lambda_{1}(-\Delta_{g_{\infty}} ; \mathrm{Neumann}) > \frac{R_{g_{\infty}}}{n-1}.$
Let $\{ g_{i} \}$ and $\{ \tilde{g}_{i} := u_{i}^{\frac{4}{n-2}} g_{i} \} \subset \mathfrak{S}_{C^{1}_{0}}$ be 
sequences each of which converges to $g_{\infty}$ with respect to the $C^{\infty}$-topology.
Then, except for a finite number of $i$, $g_{i} = \tilde{g}_{i}.$ 
\end{coro}

\begin{proof}
Since $C^{\infty}$-topology is stronger than the ILH-topology,
this claim is directly derived from Main Theorem.
\end{proof}

\begin{proof}[Outline of the proof of Theorem \ref{theo4.1}]
We can assume that $(M, g_{\infty})$ has unit volume.
Let $\{ g_{i} \} \subset \mathfrak{S}_{C^{1}_{0}}$ be a sequence which converges to $g_{\infty}$
with respect to the $C^{\infty}$-topology.
For each $i,$ 
let $u_{i} \in C^{\infty}_{+}(M)_{N}$ be a solution of relative Yamabe problem in $[g_{i}]_{0},$
that is, $\tilde{g}_{i} := u_{i}^{\frac{4}{n-2}} g_{i}$ is a relative Yamabe metric of $[g_{i}]_{0}$
with unit volume.
Since $\tilde{g}_{i} \in \mathfrak{S}_{C^{1}_{0}}$ is a relative Yamabe metric, then the following hold:
\begin{equation}
\label{vol}
|| u_{i} ||_{L^{\frac{2n}{n-2}}(g_{i})} = 1,
\end{equation}
\begin{equation}
\label{scal}
-4 \frac{n-1}{n-2} \Delta_{g_{i}} u_{i} + R_{g_{i}} u_{i} = Y(M, [g_{i}]_{0}) u_{i}^{\frac{n+2}{n-2}}~\mathrm{on}~M,
\end{equation}
\begin{equation}
\label{mean}
\nu_{g_{i}} ( u_{i}) = 0~\mathrm{on}~\partial M.
\end{equation}

By \cite{escobar1992yamabe}, the assumptions in Theorem \ref{theo4.1} implies that
\[
Y(M, [g_{\infty}]_{0}) < Y(S^{n}_{+}, [g_{\mathrm{std}}]_{0}).
\]
Since $g \mapsto Y(M, [g]_{0})$ is continuous with respect to the $C^{2}$-topology,
\[
Y(M, [g_{i}]_{0}) < Y(S^{n}_{+}, [g_{{\rm std}}]_{0})
\] 
for sufficiently large $i.$
Hence, we can apply the similar argument in the proof of Theorem 5.1 in \cite{bohm2004variational}
to that on a manifold with boundary after slight modifications.
Then, there exists a subsequaence
$\{ u_{i_{k}} \} \subset \{ u_{i} \}$ and $u_{\infty} \in C^{\infty}_{+}(M)_{N},~
\tilde{R} \in \mathbb{R}$ such that 
\[
\tilde{g}_{i} = u_{i_{k}}^{\frac{4}{n-2}} g_{i_{k}} \rightarrow \tilde{g}_{\infty} := u_{\infty}^{\frac{4}{n-2}} g_{\infty}~\mathrm{as}~k \rightarrow \infty
\]
and
\begin{equation}
\label{volinfty}
|| u_{\infty} ||_{L^{\frac{2n}{n-2}}(g_{\infty})} = 1,
\end{equation}
\begin{equation}
\label{scalinfty}
-4 \frac{n-1}{n-2} \Delta_{g_{\infty}} u_{\infty} + R_{g_{\infty}} u_{\infty} = \tilde{R} u_{\infty}^{\frac{n+2}{n-2}}~\mathrm{on}~M,
\end{equation}
\begin{equation}
\label{meaninfty}
\nu_{g_{\infty}} (u_{\infty}) = 0~\mathrm{on}~\partial M.
\end{equation}
Here, the above convergence $\tilde{g}_{i} \rightarrow \tilde{g}_{\infty} (k \rightarrow \infty)$ is the $C^{\infty}$-convergence with respect to $g_{\infty}.$ 
Then, from the regularity theorem(\cite{cherrier1984problemes}) and the maximum principle,
$u_{\infty} \in C^{\infty}_{+}(M)_{N}$.
From (\ref{volinfty}) and (\ref{scalinfty}),
we have $\tilde{g}_{\infty} \in \mathfrak{S}_{C^{1}_{0}}$.
Hence, from the uniqueness assumption for $g_{\infty}$,
$\tilde{g}_{\infty} = g_{\infty}.$
Therefore $\tilde{g}_{i_{k}} = g_{i_{k}}$
from Corollary \ref{theo4.2}, except for finite number of $k.$
\end{proof}

By using a compactness result proved by Discozi-Khuri~{\cite[Theorem~1.1]{disconzi2017compactness}},
we can also prove the following:
\begin{theo}
Let $(M^{n}, g_{\infty})$ be a compact connected Riemannian manifold  with smooth non-empty totally geodesic boundary $\partial M.$
Assume that either $3 \le n \le 7,$ or both $8 \le n \le 24$ and that $M$ is spin.
Let $g_{\infty} \in \mathfrak{S}_{C^{1}_{0}}$ be a unique relative csc metric in $[g_{\infty}]_{0}$
up to rescaling with $\lambda_{1}(-\Delta_{g_{\infty}} ; \mathrm{Neumann}) > \frac{R_{g_{\infty}}}{n-1}$.
Then each relative csc metric sufficiently close to $g_{\infty}$ in $\mathscr{M}_{C^{1}_{0}}$
with respect to he $C^{\infty}$-topology is also a relative {\it unique} csc metric
up to rescaling in its relative conformal class.
\end{theo}

\begin{proof}
In the proof of Theorem \ref{theo4.1},
we will take $\tilde{g}_{i} := u_{i}^{\frac{4}{n-2}} g_{i}$ as another relative csc metric in $[g_{i}]_{0}$
with unit volume.
Then, from the compactness result {\cite[Theorem~1.1]{disconzi2017compactness}},
there exist a subsequence
$\{ u_{i_{k}} \} \subset \{ u_{i} \}$ and $u_{\infty} \in C^{\infty}_{+}(M)_{N},~
\tilde{R} \in \mathbb{R}$ such that 
\[
\tilde{g}_{i} = u_{i_{k}}^{\frac{4}{n-2}} g_{i_{k}} \rightarrow \tilde{g}_{\infty} := u_{\infty}^{\frac{4}{n-2}} g_{\infty}~\mathrm{as}~k \rightarrow \infty
\]
and
\begin{equation*}
|| u_{\infty} ||_{L^{\frac{2n}{n-2}}(g_{\infty})} = 1,
\end{equation*}
\begin{equation*}
-4 \frac{n-1}{n-2} \Delta_{g_{\infty}} u_{\infty} + R_{g_{\infty}} u_{\infty} = \tilde{R} u_{\infty}^{\frac{n+2}{n-2}}~\mathrm{on}~M,
\end{equation*}
\begin{equation*}
\nu_{g_{\infty}} (u_{\infty}) = 0~\mathrm{on}~\partial M.
\end{equation*}
Then, 
the same argument as that in the proof of Theorem \ref{theo4.1} implies that
$\tilde{g}_{i_{k}} = g_{i_{k}}$
except for a finite number of $k$.
This completes the proof.
\end{proof}

\subsection{A characterization of relative Einstein metrics}
~~In the case of $\partial M = \emptyset,$ 
we recall the {\it Yamabe invariant} $Y(M)$ of $M$ (cf. \cite{kobayashi1987scalar}, \cite{schoen1989variational}) defined by
\[
Y(M) := \sup_{C \in \mathcal{C}(M)} Y(M, C) = \sup_{C \in \mathcal{C}(M)} \bigl( \inf_{g \in C} Y(M, C) \bigr),
\]
where $\mathcal{C}(M)$ denotes the set of all conformal classes on $M.$
By the Koiso's decomposition theorem, one can get the following:
\begin{theo}[cf.~{\cite[Proposition~5.89]{aubin2013some}}]
Let $M^{n}$ be a closed manifold of dimension $\ge 3$ and
$g$ a {\it unique} csc metric (up to rescaling) in its conformal class $[g].$
Assume that $Y(M)$ is attained by $g$ and that $\lambda_{1}(-\Delta_{g}) > \frac{R_{g}}{n-1}.$
Then, $g$ is an Einstein metric.
\end{theo}

We can prove a similar statement below on a manifold with boundary.
Let $M^{n}$ be a compact connected Riemannian manifold of dimension $n \ge 3$ with non-empty smooth boundary $\partial M$.
For each conformal class $C$ on $M,$ we define an invariant $\mathcal{Y}(M ; C)$:
\[
\mathcal{Y}(M ; C) := \sup_{\bar{C} \in \mathcal{C}_{C}} \inf_{\bar{g} \in \bar{C}_{0}} \mathcal{E}(\bar{g})
= \sup_{g \in \mathscr{M}_{C^{1}_{0}}} Y(M, [g]_{0}),
\]
where $\mathcal{C}_{C} := \bigl\{ \bar{C};~\mathrm{conformal~class~on}~M \bigm| \bar{C}||^{1}_{\partial M} = C||^{1}_{\partial M} \bigr\}$.
From the Aubin-type inequality, it holds that $\mathcal{Y}(M ; C) \le \mathcal{Y}(S^{n}_{+}, [g_{\mathrm{{\rm std}}}])$ (see \cite{akutagawa2001notes} for instance).
Then, we can prove the following:
\begin{theo}
\label{theo4.3}
Fix a conformal class $C$ on $M.$
Let $\tilde{g}$ be a unique relative csc metric (up to rescaling) in $[\tilde{g}]_{0}$
with $\tilde{g}||^{1}_{\partial M} \in C_{0}||^{1}_{\partial M}.$
Assume that $\mathcal{Y}(M ; C)$ is attained by $\tilde{g}$ and that $\lambda_{1}(-\Delta_{\tilde{g}} ; \mathrm{Neumann}) > \frac{R_{\tilde{g}}}{n-1}.$
Moreover, assume that $\mathcal{Y}(M ; C) < Y(S^{n}_{+}, [g_{{\rm std}}]_{0}).$
Then, $\tilde{g}$ is a relative Einstein metric.
\end{theo}

\begin{proof}
From Main Theorem, 
every critical point of $\mathcal{E}|_{\mathfrak{S}_{C^{1}_{0}}}$ is Einstein (cf. {\cite[Proposition~4.47]{besseeinstein}}).
Hence,
it is enough to prove that $\frac{d}{dt} \mathcal{E}(g(t))|_{t = 0} = 0$
for any smooth deformation $g(t)$ of $\tilde{g}$ in $\mathfrak{S}_{C^{1}_{0}}$.

We now remark that, under the condition $Y(M, [g]_{0}) < Y(S^{n}_{+}, [g_{{\rm std}}]_{0}),$
there exists always a relative Yamabe metric in $[g]_{0}$ (\cite{cherrier1984problemes}).
Hence, by Theorem \ref{theo4.1}, $Y(M, [g(t)]_{0}) = \mathcal{E}(g(t))$
for sufficiently small $|t| << 1.$
On the other hand, since $Y(M, [g(t)]) \le \mathcal{Y}(M ; C)$ and
$Y(M,[g(0)]_{0}) = \mathcal{Y}(M ; C)$, we have
\[
0 = \frac{d}{dt} Y(M,[g(t)])|_{t = 0} = \frac{d}{dt} \mathcal{E}(g(t))|_{t = 0}.
\]
Therefore, $\tilde{g}$ is a relative Einstein.
\end{proof}

\begin{rema}
From the characterization Theorem \ref{theo4.3} for relative Einstein metrics,
we would like to suggest that the {\it relative Yamabe invariant}

\noindent
$Y(M, \partial M, C|_{\partial M})$ defined in {\cite[Section~1]{akutagawarelative}}:
\[
Y(M, \partial M, C|_{\partial M}) := \sup_{g \in \mathscr{M}_{C_{0}|_{\partial}}} Y(M, [g]_{0})
\]
should be replaced by the above $\mathcal{Y}(M ; C) = \sup_{g \in \mathscr{M}_{C^{1}_{0}}} Y(M, [g]_{0}).$
\end{rema}

\bibliographystyle{spmpsci}
\bibliography{decomp}

\textit{E-mail adress}:~a19.fg4w@g.chuo-u.ac.jp

\textsc{Department Of Mathematics, Chuo University, Tokyo 112-8551, Japan}

\end{document}